\documentclass[11pt,reqno,xcolor=svgnames]{amsart}

\usepackage{amsmath, amsthm, amssymb, amsfonts, verbatim,tikz}

\usepackage{xspace,mathtools,bm}

\usepackage{enumitem}

\usepackage[svgnames]{xcolor}

\usepackage{cite}

\usepackage[pagebackref=true, colorlinks=true, citecolor=blue]{hyperref}

\usepackage{cleveref}
\usepackage{marginnote}
\usepackage{datetime}

\def\0{{\bf 0}}

\def\R{{\mathbb R}}
\def\Z{{\mathbb Z}}

\def\N{{\mathbb N}}

\theoremstyle{plain}

\newtheorem{theorem}{Theorem}[section]
\newtheorem{cor}[theorem]{Corollary}
\newtheorem{prop}[theorem]{Proposition}
\newtheorem{lem}[theorem]{Lemma}
\newtheorem{lemma}[theorem]{Lemma}

\theoremstyle{definition}
\newtheorem{definition}[theorem]{Definition}

\newtheorem*{theorem*}{Theorem}

\theoremstyle{plain}
\newtheorem{remark}[theorem]{Remark}

\newtheorem*{remark*}{Remark}

\theoremstyle{remark}

\usepackage{scalerel,stackengine}
\stackMath
\newcommand\reallywidehat[1]{%
\savestack{\tmpbox}{\stretchto{%
  \scaleto{%
    \scalerel*[\widthof{\ensuremath{#1}}]{\kern-.6pt\bigwedge\kern-.6pt}%
    {\rule[-\textheight/2]{1ex}{\textheight}}
  }{\textheight}%
}{0.5ex}}%
\stackon[1pt]{#1}{\tmpbox}%
}

\usepackage{tcolorbox}
\usepackage{tikz}
\tcbuselibrary{skins}
\usetikzlibrary{shadings}

\tcbset{
    myimage/.style={
        enhanced,
        overlay={
            \begin{scope}[shift={([xshift=1mm, yshift=7mm]frame.north west)}]
            \end{scope}}}}

\tcbset{
    skin=enhanced,
    fonttitle=\bfseries,
    interior style={white},
    segmentation style={black,solid,opacity=0.2,line width=1pt}}

\newtcolorbox{TitledBox}[2][]{
    myimage,              
    coltitle=black,       
    colbacktitle=white,   
    title=My title,
    attach boxed title to top center={
        yshift=-3mm,
        yshifttext=-1mm},
    attach boxed title to top left={
        xshift=1cm,
        yshift=-2mm},
    boxed title style={
        size=small},
    title={#2},#1}

\newcommand{\abs}[1]{\left\vert#1\right\vert}
\newcommand{\smallabs}[1]{\ensuremath{\vert #1 \vert}}

\DeclareMathOperator{\dv}{div}

\DeclareMathOperator{\supp}{supp}
\newcommand{\iny}{\ensuremath{\infty}}
\newcommand{\grad}{\ensuremath{\nabla}}
\newcommand{\prt}{\ensuremath{\partial}}

\newcommand{\brac}[1]{\ensuremath{\left[ #1 \right]}}
\newcommand{\pr}[1]{\ensuremath{\left( #1 \right) }}

\newcommand{\norm}[1]{\ensuremath{\left\Vert #1 \right\Vert}}

\newcommand{\smallnorm}[1]{\ensuremath{\Vert #1 \Vert}}
\newcommand{\Cal}[1]{\ensuremath{\mathcal{#1}}}
\newcommand{\wh}{\widehat}

\newcommand{\al}{\alpha}
\newcommand{\esslim}{\mathop{\mathrm{ess\,lim}}}

\newenvironment{Indent}
  {%
  \begin{list}{}%
    {\setlength{\leftmargin}{3em}%
     \setlength{\rightmargin}{0pt}}%
   \item\relax}
  {\end{list}}

\newcommand{\wstar}{\overset{*}{\rightharpoonup}}

\allowdisplaybreaks

\crefname{cor}{Corollary}{Corollaries} 
									   
\crefname{lemma}{Lemma}{Lemmas}	       

\crefname{section}{Section}{Sections}
\Crefname{section}{Section}{Sections}

\crefname{appendix}{Appendix}{Appendices}
\Crefname{appendix}{Appendix}{Appendices}

\crefname{theorem}{Theorem}{Theorems}
\Crefname{theorem}{Theorem}{Theorems}

\crefname{prop}{Proposition}{Propositions}
\Crefname{prop}{Proposition}{Propositions}

\crefname{conj}{Conjecture}{Conjectures}
\Crefname{conj}{Conjecture}{Conjectures}

\crefname{definition}{Definition}{Definitions}
\Crefname{definition}{Definition}{Definitions}

\crefname{remark}{Remark}{Remarks}
\Crefname{remark}{Remark}{Remarks}

\crefname{assumption}{Assumption}{Assumptions}
\Crefname{assumption}{Assumption}{Assumptions}

\creflabelformat{enumi}{(#1)}
\creflabelformat{enumii}{(#1)}

\crefformat{equation}{(#2#1#3)}
\crefrangeformat{equation}{(#3#1#4) through (#5#2#6)}
\crefmultiformat{equation}
    {(#2#1#3)}%
    { and~(#2#1#3)}
    {, (#2#1#3)}
    { and~(#2#1#3)}

\newcommand{\NoteToSelf}[1]{
    }

\newcommand{\Ignore}[1]{}
\newcommand{\stardot}{{\mathop{* \cdot}}}

\renewcommand{\epsilon}{\varepsilon}
\newcommand{\eps}{\ensuremath{\varepsilon}}

\newcommand{\FTF}
    {\Cal{F}}

\newcommand{\FTR}
    {\Cal{F}^{-1}}
    
\newcommand{\ol}{\overline}
\DeclareMathOperator{\Id}{Id} %
\newcommand{\Del}{\dot{\Delta}}
\newcommand{\ox}{{\ol{x}}}
\newcommand{\otheta}{\ol{\theta}}
\newcommand{\ou}{\ol{u}}

\newbool{HaveBBM}
\booltrue{HaveBBM}

\ifbool{HaveBBM}{
	\usepackage{bbm}
    }
    {
    }

\newcommand{\CharFunc}{
    \ifbool{HaveBBM}{
        \ensuremath{\mathbbm{1}}
        }
        {
        \ensuremath{\bm{1}}
        }
    }

\begin{document}
\newdateformat{mydate}{\THEDAY~\monthname~\THEYEAR}

\raggedbottom

\numberwithin{equation}{section}

\title
    [Weak Solutions to SQG]
    {Non-decaying Weak Solutions to the 2D Quasi-geostrophic Equations}

\author{David M. Ambrose}
\address{Department of Mathematics, Drexel University}
\curraddr{}
\email{dma68@drexel.edu}

\author{Ryan Aschoff}
\address{Software Development, Google}
\curraddr{}
\email{raschoff@google.com}

\author{Elaine Cozzi}
\address{Department of Mathematics, Oregon State University}
\curraddr{}
\email{cozzie@math.oregonstate.edu}

\author{James P. Kelliher}
\address{Department of Mathematics, University of California, Riverside}
\curraddr{}
\email{kelliher@math.ucr.edu}

\date{} 



\begin{abstract}
We investigate weak solutions to the two-dimensional quasi-geostrophic equations without dissipation.  We establish global existence of weak solutions for temperature bounded and lacking spatial decay and velocity in the space $L^2_{ul}(\mathbb{R}^2)$.  Our methods rely on a spectral Serfati identity, which we use to establish uniform $L^2_{ul}$ bounds on a sequence of velocities satisfying the dissipative equations.  These bounds, combined with a maximum principle on the scalar temperature, allow us to pass to the zero-dissipation limit, giving global-in-time weak solutions.            
\end{abstract}

\maketitle

\tableofcontents

\section{Introduction}

    \noindent In \cite{AACK25}, the current authors established the well-posedness of non-decaying regular viscous solutions to the SQG equations in the plane with fractional subcritical diffusion, which includes the following case of classical diffusion (see \cref{T:SQG1} below):
    \begin{align}\label{e:SQGnu}
        \tag{$SQG_\nu$}
        \begin{cases}
            \prt_t \theta + u\cdot\grad\theta = \nu \Delta\theta
                &\text{in }[0,T] \times \R^2,\\
            \Del_j u = (\Del_j K) *\theta
                &\text{in } [0,T]\times \R^2, \, \forall\,j\in\Z,\\
            \theta|_{t = 0} = \theta_0, \, u|_{t = 0} = u_0
                &\text{in }\R^2.
        \end{cases}
    \end{align}
    The scalar function $\theta$, often associated with the temperature of a fluid, is transported by the velocity field $u$. $K$ is the kernel for $\grad^\perp (-\Delta)^{-\frac{1}{2}}$, formally a Riesz operator. The Littlewood-Paley operator $\Del_j$ isolates a specific range of frequencies, and $\Del_j u = (\Del_j K) *\theta$ for all $j$ is a form of the constitutive law, which we refer to as the \textbf{homogeneous constitutive law}, appropriate for non-decaying data. For $\theta$ having sufficient decay, say $\theta \in L^2$, the homogeneous constitutive law follows from the now classical form of the constitutive law, $u = K * \theta = \grad^\perp (-\Delta)^{-\frac{1}{2}} \theta$.

    Here, we leverage this result from \cite{AACK25} to obtain, via a vanishing viscosity argument, the existence, for non-decaying data, of weak solutions to the inviscid SQG equations,
    \begin{align}\label{e:SQG0}
        \tag{$SQG_0$}
        \begin{cases}
            \prt_t \theta + u\cdot\grad\theta = 0
                &\text{in }[0,T] \times \R^2,\\
            \Del_j u = (\Del_j K) *\theta
                &\text{in } [0,T]\times \R^2, \, \forall\,j\in\Z,\\
            \theta|_{t = 0} = \theta_0, \, u|_{t = 0} = u_0
                &\text{in }\R^2.
        \end{cases}
    \end{align}
    This leads to our main result, \cref{T:MainResult}, which we prove in \cref{S:Existence}. The equivalent for decaying data was first proved by Resnick in his PhD thesis \cite{Resnick}, and later, using a different proof, by Marchand in \cite{Marchand}. Our approach is closer to that of Marchand.

    \begin{theorem}\label{T:MainResult}
    Assume the pair $( \theta_0,u_0)$ belongs to $ L^{\infty}(\R^2)\times L^2_{ul}(\R^2)$ and satisfies $\dot{\Delta}_j u_0 = (\dot{\Delta}_jK)\ast \theta_0 \text{ for all }j\in\Z$, div $u_0=0$.  Then for every $T>0$, there exists a non-unique weak solution $(\theta,u)$ to \cref{e:SQG0} on $[0,T]$ with initial data $( \theta_0,u_0)$. %
    Furthermore, the following Serfati identity (discussed below) holds:
    \begin{equation}\label{e:SerfatiWeak}
        \begin{split}
            u(t, x) - u_0(x) 
                &= -\int_0^t \left(S_{N} \nabla K \stardot (\theta u) \right)(s,x) \, ds
                + (H_NK) \ast (\theta(t) - \theta_0)(x).
        \end{split}
    \end{equation}
    \end{theorem}
\begin{remark}
Given $(\theta_0,u_0)$ satisfying the assumptions of \cref{T:MainResult}, $T>0$, and a continuous function $U_{\infty}:[0,T]\rightarrow \R^2$, one can show more generally that there exists a weak solution $(\theta,u)$ satisfying (\ref{e:SerfatiWeak}) with $U_{\infty}(t)$ added to the right hand side.  See \cref{R:InviscidSerfatiID} for details.   
\end{remark}
 
    Our weak solutions are defined as follows:

    \begin{definition}\label{weaksolutiondef2}
    Let $T>0$.  A weak solution to \cref{e:SQG0} on $[0,T] \times \R^2$ is a pair 
    \begin{equation}\label{regularitycondition}
    (\theta, u) \in L^{\infty}([0,T]; L^{\infty}(\R^2))\times L^{\infty}([0,T]; L^2_{ul}(\R^2))
    \end{equation}
    which, on all of $[0, T] \times \R^2$, has $\dv u = 0$ and  satisfies the homogeneous constitutive law, and for all $\phi\in C_c^\iny([0,T]\times \R^2)$, 
    \begin{equation}\label{e:WeakId}
    \begin{split}
    &\int_{0}^T\int_{\R^2} \partial_t \phi \, \theta \, dx \, dt  +\int_{0}^T\int_{\R^2} \nabla \phi \cdot u\theta \, dx \, dt = \int_{\R^2} \phi(T) \theta(T)\, dx-\int_{\R^2} \phi(0) \theta_0\, dx. 
    \end{split}
    \end{equation}
    \end{definition}

    \begin{remark}\label{R:WeakerTestFunctions}
        We require our test functions in \cref{weaksolutiondef2} to lie in $C_c^\iny([0,T]\times \R^2)$ for later convenience, but requiring that they lie in $C^1_c([0, T) \times \R^2)$  yields an equivalent definition. Allowing a change in $\theta(t)$ on a set of measure zero in $[0, T]$, an equivalent form of weak solution---that used, for instance, by Resnick in \cite{Resnick}---is obtained by replacing the condition in \cref{e:WeakId} with requiring that for all $\phi \in C_c^\iny(\R^2)$ and all $t \in [0, T]$,
        \begin{align*}
            \int_{\R^2} \phi \theta(t) \, dx
                - \int_0^t \int_{\R^2} \grad \phi \cdot u \theta \, dx \, dt
                = \int_{\R^2} \phi \theta_0 \, dx.
        \end{align*} 
        \cref{R:TimeContinuity} describes the time continuity of $\theta$ and $u$.
    \end{remark}

\smallskip\noindent\textbf{Our approach.}
Our proof of existence in \cref{T:MainResult} proceeds in three  phases:
\begin{Indent}
    \begin{enumerate}[label=\Roman*, align=left]
        \item Establish a Serfati identity separating the solutions into low and high frequencies.
    
        \item Identify a candidate weak solution as a limit of a subsequence of smooth solutions to \cref{e:SQGnu} as $\nu \to 0$.
    
        \item Prove that the candidate solution is, in fact, a solution.
    \end{enumerate}
\end{Indent}

In more detail, the phases are as follows:

\smallskip\noindent\textbf{Phase I.} We obtain a \textit{Serfati identity}, \cref{FourierSerfati2}, for the classical solutions to \cref{e:SQGnu} obtained in \cite{AACK25}. This type of identity was first employed by Serfati in \cite{Serfati} to obtain existence and uniqueness of solutions to the 2D Euler equations with both the velocity and vorticity bounded. In its original form (see also \cite{AKLN2015}), the identity was based on a decomposition in space of a solution to the 2D Euler equations, allowing separate control of the near field and far field of the Biot-Savart kernel. This was also the approach taken in \cite{AKLN2015,CozziKelliher2019,KBounded,Ambrose2025,Ciampa2022,CiampaCrippaSpirito2021}.

We employ a parallel approach, decomposing the solutions to \cref{e:SQGnu} into low and high frequencies using classical Littlewood-Paley operators. This spectral method of  decomposition works more efficiently with the underlying function spaces we employ, particularly $L^2_{ul}$, the space in which the velocity for our weak solutions lie. We note that in \cite{Taniuchi} p. 178, Taniuchi, though not explicitly using a Serfati identity, uses a spectral decomposition of the velocity formulation of solutions to the 2D Euler equations for somewhat similar purposes.

Whether decomposed by physical space or frequency space, an identity like that of Serfati in \cite{Serfati} we refer to as a \textit{Serfati identity}. Although designed for, and especially suited to, non-decaying data, these types of identities are also quite useful in applications to decaying solutions and, at least for spatial decompositions, solutions in domains.

\bigskip

\smallskip\noindent\textbf{Phase II.} We employ the Serfati identity from Phase I along with the results of \cite{AACK25} to establish sufficient control on smooth, non-decaying solutions to \cref{e:SQGnu}.  From this control we obtain, in the limit as $\nu \to 0$, a candidate weak solution to \cref{e:SQG0}. We use non-decaying solutions to \cref{e:SQGnu}, given by \cref{T:SQG1}, because these solutions make it possible to obtain uniform bounds on the smoothed initial data (as in \cref{prop:convergence-of-initial-data-u}).  The following theorem is proved in \cite{AACK25} (here, $C^k_b(\R^2)$ denotes the space of $k$-times differentiable functions on $\R^2$ with continuous, bounded derivatives up to and including order $k$):

\begin{theorem}[\cite{AACK25}]\label{T:SQG1}
Let $k\geq 2$ be an integer. Assume $\theta_0$, $u_0$ belong to $C^k_b(\R^2)$ and $(\theta_0,u_0)$ satisfies the homogeneous constitutive law, with div $u_0=0$.  Then for each $T>0$, there exists a classical solution ($\theta,u$) to \cref{e:SQGnu} in $C^1( (0,T]; C^k_b(\R^2) ) \times C( (0,T]; C^k_b(\R^2) )$. Moreover, $\theta$ satisfies
\begin{equation}\label{max}
\| \theta (t) \|_{L^{\infty}} \leq \|\theta_0 \|_{L^{\infty}} \quad \text{for all } t\in [0,T],
\end{equation}
($\theta,u$) satisfies the homogeneous constitutive law on $(0,T]\times \R^2$, and div $u=0$ on $(0,T]\times \R^2$. 
\end{theorem}

\smallskip\noindent\textbf{Phase III.}
As is usually the case, the main difficulty is showing that the nonlinear term of the approximating solutions converges to that of the candidate solution. The key to this is \cref{L:ffInt}, which we employ in the proof of \cref{P:convergence2}. \cref{L:ffInt} is analogous to Marchand's Lemma 2.1 in \cite{Marchand} and to (2.11) in \cite{Resnick}; both are similarly used to establish convergence of the nonlinear term.

\smallskip\noindent\textbf{A brief history of SQG.}
The non-diffusive SQG equations, \cref{e:SQG0} with the constitutive law $u = K * \theta = \grad^\perp (-\Delta)^{-\frac{1}{2}} \theta$, were introduced by Constantin, Majda, and Tabak \cite{ConstantinMajdaTabak1994A}, further expounded upon by the same authors in \cite{ConstantinMajdaTabak1994B}, to model atmospheric fluid flows and as a 2D analogy to the 3D Euler equations. This excited a great deal of work among theoreticians to investigate analytically properties of solutions to SQG with and without diffusion: existence of weak solutions, well-posedness of smooth solutions, decay of solutions in time, 
etc.. It is impossible to give even a limited account here of the full range of activity; we mention only those results having a more direct impact on the present work.

Resnick \cite{Resnick} was the first to establish the existence of weak solutions to the non-diffusive SQG equations, for $\theta \in L^\iny(0, T; L^2(\R^2))$, though our proof of existence for non-decaying data is closer to that of Marchand in \cite{Marchand}. Resnick also established the maximum principle, $\norm{\theta(t)}_{L^p} \le \norm{\theta_0}_{L^p}$, $1 < p \le \iny$, for smooth solutions to dissipative SQG on the torus. The decay of $\norm{\theta(t)}_{L^p}$ over time, including $p = \iny$, for diffusive equations was demonstrated by C\'ordoba and C\'ordoba \cite{cordoba2004maximum}.

The vanishing viscosity limit, in which the diffusive parameter ($\nu$ here) is taken to zero, has been studied in two ways. In one, such as in \cite{Wu1997,Berselli}, the solution to non-diffusive SQG is already known, and the convergence of the diffusive to non-diffusive SQG is demonstrated, typically with a rate of convergence. In the other, uniform-in-$\nu$ control is obtained on the solutions to the diffusive equations, and a compactness argument of some sort leads to the existence of solutions to the non-diffusive equations. This approach is taken by Marchand \cite{Marchand} and here, as well as in the recent \cite{LuigiLatocca2026}.

\smallskip\noindent\textbf{Organization of the paper.} In \cref{S:preliminary}, we define the Littlewood-Paley operators, establish a number of lemmas related to the Riesz kernel $K$ that we will use throughout the rest of the paper, and explore the homogeneous constitutive law. We prove the non-uniqueness of weak solutions to \cref{e:SQG0} in \cref{S:NonUniquess}. We implement Phases I, II, and III in \cref{S:Apriori}, \ref{S:SolutionCandidate}, and \ref{S:Existence}, respectively.

\section{Definitions and preliminary lemmas}\label{S:preliminary}
In this section, we give some notation, definitions, and lemmas that will be useful in what follows.

We set $$\Lambda = (-\Delta)^{1/2}.$$  We use $\Phi$ to denote the fundamental solution of the fractional Laplacian $\Lambda$ on $\R^2$; that is,    
\begin{equation*}
	\Phi(x) = \frac{1}{2 \pi |x|},
\end{equation*}
and we set $$K=\nabla^{\perp}\Phi.$$

\begin{definition}
We define the space $L^2_{ul} (\R^d)$ by
$$L^2_{ul} (\R^d) := \{ f\in \mathcal{S}'(\R^d) : \|f\|_{L^2_{ul}} <\infty\},$$
where
\begin{equation}
    \|f\|_{L^2_{ul}} := \sup_{x\in \R^d} \left(\int_{|x-y|<1} |f(y)|^2 dy\right)^{1/2}.
\end{equation}
\end{definition}
\begin{definition}\label{H}
Let $R>0$.  Define the set $H_R(\R^d)$ to be the set of $\eta\in C^\infty_c(\R^d)$ for which $\supp (\eta) \subset B_R(x)$ for some $x$ and $\|\eta\|_{L^{\infty}} \leq 1$.
\end{definition}
It is easy to see that for each $R>0$, there is the equivalence of norms
$$\|f\|_{L^2_{ul}} = \sup_{\eta\in H_R} \|\eta f \|_{L^2}.$$
\subsection{The Littlewood-Paley operators}\label{S:LPOperators}
We now give an overview of the Littlewood-Paley operators and some of their properties.  It is classical that there exist two functions ${\chi}, {\varphi} \in \mathcal{S}(\R^d)$ with supp $\wh{\chi}\subset \{\xi\in \R^d: |\xi |\leq \frac{5}{6} \}$ and supp $\wh{\varphi}\subset \{\xi\in \R^d: \frac{3}{5} \leq|\xi |\leq \frac{5}{3} \}$, such that, if for every $j\in\Z$ we set $\varphi_j(x) = 2^{jd} \varphi(2^j x)$, then
\begin{equation*}
\begin{split}
	&\wh{\chi}+ \sum_{j=0}^{\infty} \wh{\varphi_j}
		= \wh{\chi} + \sum_{j=0}^{\infty} \wh{\varphi}(2^{-j} \cdot) 
		\equiv 1.
\end{split}
\end{equation*}

For $f\in \mathcal{S}'(\R^d)$ and $j\in\Z$, 
define the homogeneous Littlewood-Paley operators $\dot{\Delta}_j$ by
\begin{equation*}
    \dot{\Delta}_j f = {\varphi}_j \ast f.
\end{equation*}

For $n\in\Z$, define ${\chi}_n \in \mathcal{S}(\R^d)$ in terms of its Fourier transform ${\wh{\chi}}_n$, where ${\wh{\chi}}_n$ satisfies 
\begin{equation*}
{\wh{\chi}}_n (\xi) =   1 - \sum_{j=n }^{\infty} \wh{\varphi}_j(\xi)
\end{equation*}
for all $\xi\in\R^d$. For $f\in\mathcal{S}'(\R^d)$, define the operator $S_n$ by  
\begin{equation*}
S_n f = {\chi}_n \ast f.
\end{equation*}
For each $n\in\Z$, we set $$H_n = Id - S_n.$$

We will make use of Bernstein's Lemma in what follows.  A proof of the lemma can be found in \cite{Chemin1}, Chapter 2.  Below, $C_{a,b}(0)$ denotes the annulus with inner radius $a$ and outer radius $b$ centered at the origin.  
\begin{lem}\label{bernstein}
(Bernstein's Lemma) Let $r_1$ and $r_2$ satisfy $0<r_1<r_2<\infty$, and let $p$ and $q$ satisfy $1\leq p \leq q \leq \infty$. There exists a positive constant $C$ such that for every integer $k$, if $u$ belongs to $L^p(\R^d)$, and supp $\wh{u}\subset B_{r_1\lambda}(0)$, then 
\begin{equation}\label{bern1}
\sup_{|\alpha|=k} ||\partial^{\alpha}u||_{L^q} \leq C^k{\lambda}^{k+d(\frac{1}{p}-\frac{1}{q})}||u||_{L^p}.
\end{equation}
Furthermore, if supp $\wh{u}\subset C_{r_1\lambda, r_2\lambda}(0)$ then 
\begin{equation}\label{bern2}
C^{-k}{\lambda}^k||u||_{L^p} \leq \sup_{|\alpha|=k}||\partial^{\alpha}u||_{L^p} \leq C^{k}{\lambda}^k||u||_{L^p}.
\end{equation} 
\end{lem}

We now state a lemma giving boundedness of Littlewood-Paley operators on $L^2_{ul}(\R^d)$.  
\begin{lemma}\label{SnBound}
Assume $f$ belongs to $L^2_{ul}(\R^d)$.  There exists a constant $C>0$ such that for all $n, j\in\Z$,
    \begin{equation*}
        \| S_n f \|_{L^2_{ul}} \leq C\|f \|_{L^2_{ul}},\quad
        \| H_n f \|_{L^2_{ul}} \leq C\|f \|_{L^2_{ul}},\quad
        \| \dot{\Delta}_j f \|_{L^2_{ul}} \leq C\|f \|_{L^2_{ul}}.
    \end{equation*}
\end{lemma}
\begin{proof} A proof of the first inequality in \cref{SnBound} can be found in \cite{ACEK}, from which the estimate for $H_n$ follows immediately.  The proof of the final inequality in \cref{SnBound} is identical to the proof for $S_n$. \end{proof}

\subsection{Riesz kernel estimates}  We establish several useful estimates involving the Riesz kernel and Littlewood-Paley operators.
\begin{lemma}\label{L:DelKxBound}
    For all $j \in \Z$,
	$\smallnorm{x \otimes \Del_j K(x)}_{L^1},
    \smallnorm{\abs{x} \Del_j K(x)}_{L^1} \le C 2^{-j}$.
\end{lemma}
\begin{proof}
	Using that $\wh{K}(\xi) = \wh{K}(2^{-j} \xi)$ and setting
	$\eta = 2^{-j} \xi$, for $n = 1, 2$,
	\begin{align*}
		\FTF &\pr{\Del_j K^n(x) x}(\xi)
			= i \grad_\xi \FTF \pr{\Del_j K^n}(\xi)
			= i \grad_\xi \pr{\wh{\varphi_j}(\xi) \wh{K^n}(\xi)} \\
			&= i \grad_\xi \pr{\wh{\varphi}(2^{-j} \xi)
				\wh{K^n}(2^{-j} \xi)}
			= 2^{-j} g(\eta),
	\end{align*}
	where $g(\eta) := i \grad_\eta (\wh{\varphi} \wh{K^n})(\eta)
	\in \Cal{S}(\R^2)$. Hence, $\FTR (g) \in \Cal{S}(\R^2) \subseteq L^1(\R^2)$, and
	\begin{align*}
		\smallnorm{x \otimes \Del_j K(x)}_{L^1}
            &\le \sum_{n = 1}^2 \smallnorm{\Del_j K^n(x) x}_{L^1}
            \le 2^{-j} \smallnorm{\FTR (g)}_{L^1}
			= C 2^{-j}.
	\end{align*}

    The radial symmetry of $\varphi_j$ preserves the property that $K(x) \cdot x^\perp = \abs{x} \abs{K(x)}$; that is, $\Del_j K(x) \cdot x^\perp = \abs{x} \smallabs{\Del_j K(x)}$. Hence, the second inequality follows from the first.
\end{proof}

\begin{lemma}\label{LPRieszcommutator}
Given $\psi\in C^{\infty}_c(\R^2)$ and $f\in L^{\infty}(\R^2)$, for every $j\in \Z$,
\begin{equation*}
\begin{split}
&\|[\psi, {\dot{\Delta}}_j] f\|_{L^{\infty}} \leq C2^{-j} \|f\|_{L^{\infty}},\\
&\|[\psi, ({\dot{\Delta}}_jK)\ast] f\|_{L^{\infty}} \leq C2^{-j} \|f\|_{L^{\infty}}.
\end{split}
\end{equation*}
\end{lemma}
\begin{proof}
For the first inequality, note that for any function $f\in L^{\infty}(\R^2)$, using a change of variables,
\begin{equation}\label{term2}
\begin{split}
&\abs{ [\psi, \dot{\Delta}_j] f(x) }= \abs{ \int_{\R^2} \varphi_j (x-y)(\psi(x) - \psi(y))f(y) \, dy }\\
&\qquad \leq C\|f \|_{L^{\infty}}\int_{\R^2} |\varphi_j (x-y)||x-y| \, dy \leq C2^{-j}\|f\|_{L^{\infty}}. 
\end{split}
\end{equation}

For the second inequality, note that by Taylor's formula, 
\begin{equation}\label{IIIest}
\begin{split}
&|[\psi, (\dot{\Delta}_jK)\ast] f(x)| = \bigg| \int_{\R^2} {\dot{\Delta}}_j K(x-y)(\psi(x) - \psi(y)) f(y) \, dy \bigg|  \\
&\leq  \int_{\R^2}\abs{{\dot{\Delta}}_jK(x-y)f(y) \sum_{m=1}^2 \int_0^1 (\partial_m \psi) (x+ \tau(y-x))(x_m - y_m) \, d\tau } \, dy\\
&\leq \|\nabla\psi\|_{L^{\infty}} \| f \|_{L^{\infty}} \sum_{m=1}^2\int_{\R^2} \abs{{\dot{\Delta}}_jK(x-y)(x_m-y_m)} \, dy\\
&\leq \|\nabla\psi\|_{L^{\infty}} \| f \|_{L^{\infty}}\sum_{m=1}^2 \int_{\R^2} \abs{{\dot{\Delta}}_jK(y)y_m} \, dy.
\end{split}
\end{equation}
\cref{L:DelKxBound} then gives the second inequality.
\end{proof}

\begin{lemma}\label{SNKbound} 
Let $k\in\Z$.  There exists $C>0$ such that for any multi-index $\alpha$ with $|\alpha|>0$,
\begin{equation*}
\begin{split}
&\| {\partial}^{\alpha}_x S_k K\|_{L^{1}} \leq C2^{k|\alpha|},\\
\end{split}
\end{equation*}
while for any multi-index $\beta$ with $|\beta|\geq 0$,
\begin{equation*}
\begin{split}
&\| {\partial}^{\beta}_x \dot{\Delta}_k K\|_{L^{1}} \leq C2^{k|\beta|}.
\end{split}
\end{equation*}
\end{lemma}
\begin{proof}
By the definition of $S_{k}$,
\begin{equation*}
\begin{split}
&\mathcal{F}({\partial}^{\alpha}_x S_{k}K) = 2^{k|\alpha|}\widehat{{\partial}^{\alpha}_x\chi} (2^{-k}\xi)\wh{K}(2^{-k}\xi), 
\end{split}
\end{equation*} 
where we used that $\wh{K} (\xi)= \wh{K} (2^{-k}\xi)$.  Taking the inverse Fourier transform gives
$$ {\partial}^{\alpha}_x S_{k}K(x) =  2^{k|\alpha|} [2^{2k}({\partial}^{\alpha}_x\chi \ast K) (2^{k}x)].$$ 
Letting $a:\R^2\rightarrow\R$ denote a smooth, compactly supported bump function, we have   
\begin{equation}\label{lowfreqkernelL1}
\begin{split}
&\| {\partial}^{\alpha}_x \chi \ast K\|_{L^1} = \| {\partial}^{\alpha}_x \chi \ast \nabla^{\perp}\Phi \|_{L^1}\\
&\qquad \leq \| {\partial}^{\alpha}_x\nabla^{\perp}\chi \|_{L^1} \| a\Phi\|_{L^1} + \| \chi \|_{L^1} \|{\partial}^{\alpha}_x\nabla^{\perp}(1-a)\Phi \|_{L^1}  \leq C, 
\end{split}
\end{equation}
so $\| {\partial}^{\alpha}_x S_{k}K \|_{L^1} \leq C2^{k|\alpha|}$.
The proof of the second inequality is identical when $|\beta|>0$, with $\chi$ replaced by $\varphi$.  When $|\beta|=0$, we apply \cref{bern2} to write 
$$\| \dot{\Delta}_k K\|_{L^{1}} \leq C2^{-k}\| \dot{\Delta}_k \nabla K\|_{L^{1}} \leq C.$$
\end{proof}

\begin{lemma}\label{L:WeakStarSectionalConvergence}
    Suppose that  $(f_k)$ is a sequence in $C_b([0, T] \times \R^2)$ with
    \begin{align*}
        f_k \wstar f \text{ in } L^\iny([0, T] \times \R^2)
    \end{align*}
    for some $f \in L^\iny([0, T] \times \R^2)$, with the additional property that for all $h \in C_c^2(\R^2)$ there exists $C_h > 0$ such that
    \begin{align*}
        \abs{\int_{\R^2} h(x) (f_k(t, x) - f_k(s, x)) \, dx}
            \le C_h \abs{t - s}
            \text{ for all } k \in \N, t, s \in [0, T].
    \end{align*}
    Then 
    \begin{align*}
        f_k(t) \wstar f(t)
            \text{ in } L^\iny(\R^2)
            \text{ for all } t \in [0, T].
    \end{align*}
        Moreover, after changing $f(t, \cdot)$ on a set of measure zero in $[0, T]$, for all $s, t \in [0, T]$,
    \begin{align}\label{e:WeakTimeContinuityOff}
        \begin{split}
            &\int_{\R^2} h(x) f(s, x) \, dx
                \to \int_{\R^2} h(x) f(t, x) \, dx
                \text{ for all } h \in L^1(\R^2), \\
            &\abs{\int_{\R^2} h(x) (f(s, x) - f(t, x)) \, dx}
                 \le C_h \abs{t - s}
                 \text{ for all } h \in C_c^2(\R^2).
            \end{split}
    \end{align}
\end{lemma}
\begin{proof}
    For any $h \in L^1(\R^2)$, define
    \begin{align*}
        F_k(t) := \int_{\R^2} h(x) f_k(t, x) \, dx.
    \end{align*}
    Then $f_k(t) \wstar f(t) \text{ in } L^\iny(\R^2)$ if and only if $F_k(t) \to  \int_{\R^2} h(x) f(t, x)$ for all such $h \in L^1(\R^2)$.

    We start by assuming that $h \in C_c^2(\R^2)$.

    Because any weak-$*$ convergent sequence must be norm-bounded, for some $C_0 > 0$,
    \begin{align*}
        \norm{f_k}_{L^\iny([0, T] \times \R^2)} \le C_0, \,
        \norm{f}_{L^\iny([0, T] \times \R^2)} \le C_0
    \end{align*}
    for all $k$, and so both
    \begin{align*}
        \abs{F_k(t)}
            &\le \norm{h}_{L^1} \norm{f_k(t)}_{L^\iny}
            \le C \text{ and} \\
        \abs{F_k(t) - F_k(s)}
            &\le C_h \abs{t - s}
            \text{ for all } t, s \in [0, T].
    \end{align*}
    That is, $(F_k)$ is a bounded equi-Lipschitz family on $[0, T]$. Hence by
    Arzel\`{a}-Ascoli, every subsequence of $(F_k)$ has a further subsequence converging
    uniformly on $[0, T]$ to some $F$ with the same Lipschitz constant $C_h$ on $[0, T]$.

    We can identify $F$ explicitly by observing that
    for any $\psi \in L^1([0, T])$, $\psi(t) h(x)$ is in $L^1([0, T] \times \R^2)$, so $f_k \wstar f$ in $L^\iny([0, T] \times \R^2)$ gives
    \begin{align*}
        \int_0^T &\psi(t) F_k(t) \, dt
            = \int_0^T \int_{\R^2} [\psi(t) h(x)] f_k(t, x) \, dx \, dt \\
            &\to 
            \int_0^T \int_{\R^2} [\psi(t) h(x)] f(t, x) \, dx \, dt
            = \int_0^T \psi(t) F(t) \, dt
    \end{align*}
    for some further subsequence, where
    \begin{align*}
        F(t) = \int_{\R^2} h(x) f(t, x) \, dx.      
    \end{align*}
    Hence, the full sequence $(F_k)$ converges (uniformly) to $F$---as long as $h \in C_c^2(\R^2)$. This gives \cref{e:WeakTimeContinuityOff}$_2$.

    Now assume only that $h \in L^1(\R^2)$, and choose a sequence $(h_n)$ in $C_c^2(\R^2)$ converging to $h$ in $L^1$. Then
    \begin{align*}
        &\abs{F_k(t) - F(t)}
            = \abs{\int_{\R^2} h(x) (f_k(t, x) - f(t, x)) \, dx} \\
            &\qquad
            \le \abs{\int_{\R^2} (h(x) - h_n(x)) (f_k(t, x) - f(t, x)) \, dx}
                + \abs{\int_{\R^2} h_n(x) (f_k(t, x) - f(t, x)) \, dx} \\
            &\qquad
            \le 2 C_0 \norm{h - h_n}_{L^1(\R^2)}
                + \abs{\int_{\R^2} h_n(x) (f_k(t, x) - f(t, x)) \, dx}.               
    \end{align*}
    Let $\eps > 0$. First fix $n$ large enough that $2 C_0 \norm{h - h_n}_{L^1(\R^2)} < \eps/2$, then choose $k$ large enough that the final integral above is less than $\eps/2$, which we know is possible since $h_n \in C_c^2(\R^2)$.
    We conclude that $F_k(t) \to F(t)$ for all $h \in L^1(\R^2)$. As noted earlier, this means that $f_k(t) \wstar f(t) \text{ in } L^\iny(\R^2)$---for all $t \in [0, T]$.
     (A bound similar to this gives the weak time regularity of $f$ in \cref{e:WeakTimeContinuityOff}$_1$.)
\end{proof}

\begin{lemma}\label{L:ffInt}
    Fix $N \in \Z$ and for any $\phi \in C_c^\iny(\R^2)$, define the operators $A$, $B$ on scalar functions $f$, $g$ by,
    \begin{align}\label{e:ABOperators}
        \begin{split}
    	A(f, g)
			&:= \int_{\R^2} f(x)
    	    \grad \phi(x) \cdot ((H_N K) * g(x)) \, dx, \\
        B(f, g)
            &= \int_{\R^2} (H_N K) \stardot (f (\grad \phi(\cdot) - \grad \phi(x)))
                    g(x) \, dx.
        \end{split}
    \end{align}
    We have,
	\begin{align}\label{e:ffIntId}
        \begin{split}
            A(f, g)
                &= B(f, g)
                    - A(g, f), \\
            A(f, f)
                &= \frac{1}{2}B(f, f).
            \end{split}
    \end{align}
    Moreover, for a constant $C$ depending on $\phi$, for $f, g \in L^\iny(\R^2)$,
	\begin{align}\label{e:ffBound}
        \begin{split}
        \norm{B(f, g)}_{L^\iny}
            &\le C \norm{f}_{L^\iny} \norm{g}_{L^\iny} 2^{-N}, \\
		\norm{A(f, f)}_{L^\iny}
			&\le C \norm{f}_{L^\iny}^2 2^{-N}.
        \end{split}
	\end{align}
\end{lemma}
\begin{proof}
    Define $A_j, B_j$ as in \cref{e:ABOperators} but with $\Del_j$ replacing $H_N$, so that $A = \sum_{j = N}^\iny A_j$, $B = \sum_{j = N}^\iny B_j$.
    We have, $(\Del_j K)(-x) = - (\Del_j K)(x)$
	because $K(-x) = -K(x)$ and $K^n$, $n = 1, 2$, and $\varphi_j$ are radially symmetric.
    Hence,
    \begin{align*}
        A_j(f, g)
             &= \int_{\R^2} \int_{\R^2} f(x) \grad \phi(x) \cdot 
                    (\Del_j K)(x - y) g(y) \, dy
				\, dx \\
            &= \int_{\R^2} \brac{\int_{\R^2}
                    (\Del_j K^n)(x - y) f(x) (\prt_n \phi(x) - \prt_n \phi(y)) \, dx}
                    g(y) dy
				\, dx \\
            &\qquad\qquad
                + \int_{\R^2} \brac{\int_{\R^2}
                    (\Del_j K^n)(x - y) f(x) \, dx} \prt_n \phi(y)
                    g(y) dy
				\, dx \\
            &= B_j(f, g) - A_j(g, f).
    \end{align*}
    Summing over $j$ gives \cref{e:ffIntId}$_1$, and \cref{e:ffIntId}$_2$ follows from setting $g = f$.
	
    Now,
	\begin{align*}
		\Del_j K(x) = 2^{2j} (\Del_0 K)(2^j x),
	\end{align*}
	and $\Del_0 K \in \Cal{S}(\R^2)$, which decays at infinity
	faster than the reciprocal of any polynomial. In particular,
	this means that for some constant $C > 0$, for all $\abs{x} \ge 1$,
	\begin{align*}
		\smallabs{\Del_j K(x)}
			&\le C \frac{2^{2j}}{\abs{2^j x}^3}
			= C 2^{-j} \abs{x}^{-3}.
	\end{align*}
	Also, 
	\begin{align*}
		\abs{(\grad \phi(t, x) - \grad \phi(t, y)) (\Del_j K)(x - y)}
			&\le \norm{\phi}_{C^2([0, T] \times \R^2)}
				\abs{x - y} \smallabs{(\Del_j K)(x - y)}.
	\end{align*}
    Hence, letting $B_R$, $R \ge 1$,  be a ball centered at the
	origin containing
	$\supp \grad \phi$, we have
	\begin{align*}
		\abs{B(f, g)}
			&\le \frac{1}{2} \norm{f}_{L^\iny} \norm{g}_{L^\iny}
			\left\{C \int_{B_{2R} \times B_{2R}} \abs{x - y}
				\smallabs{(\Del_j K)(x - y)}
				\, dy \, dx
				\right. \\
			&\qquad\qquad\qquad
				\left.
				+ \, \, C 2^{-j}
					\int_{(B_R \times B_{2R}^C)
						\cup (B_{2R}^C \times B_R)}
					\frac{1}{\abs{x - y}^3} \,dy \, dx
			\right\}.
	\end{align*}
	For the first integral, we make the
	change of variables, $w = x - y$, $z = x + y$
	(the Jacobian being $1/2$), giving
	\begin{align*}
		\int_{B_{2R} \times B_{2R}} &\abs{x - y}
				\smallabs{(\Del_j K)(x - y)}
				\, dy \, dx
			\le \frac{1}{2} \int_{B_{4R}}
				\int_{B_{4R}}
				\abs{w} \smallabs{\Del_j K(w)} \, dw \, dz \\
			&\le \pi (4R)^2
				\smallnorm{\abs{w} \smallabs{\Del_j K}(w)}_{L^1(\R^2)}
			\le C 2^{-j},
	\end{align*}
    where we applied \cref{L:DelKxBound} in the final inequality.
	Also,
	\begin{align*}
		&\int_{(B_R \times B_{2R}^C) \cup (B_{2R}^C \times B_R)}
			\frac{1}{\abs{x - y}^3} \,dy \, dx
			= 2
				\int_{B_R} \int_{B_{2R}^C}
					\frac{1}{\abs{x - y}^3} \,dy \, dx \\
			&\qquad
			\le 16 \int_{B_R} \, dx \int_{B_R^C}
					\frac{1}{\abs{y}^3} \,dy
			\le C R,
	\end{align*}
	where we used the rough estimate on $B_R \times B_{2R}^C$
	that $\abs{x - y} \ge \abs{\abs{y} - \abs{x}} \ge
	\abs{y} - \abs{x} \ge \abs{y} - R \ge \abs{y}/2 \ge R$.
    Summing over $j$ gives \cref{e:ffBound}$_1$, and setting $g = f$ gives \cref{e:ffBound}$_2$.
\end{proof}

\noindent\subsection{The homogeneous constitutive law}
\begin{definition}
We say $(\theta, u)$ satisfies the \textit{homogeneous constitutive law} if
\begin{align*}
	\Del_j u
		= (\Del_j K) * \theta
		\text{ for all } j \in \Z.
\end{align*}
\end{definition}
In \cref{SNulimit}, we establish an important property of a pair $(\theta, u)$ satisfying the homogeneous constitutive law.

\begin{lemma}\label{SNulimit}
Assume $u\in L^{\infty}([0,T]; L^2_{ul}(\R^2))$ satisfies the homogeneous constitutive law with $\theta$ at each $t\in [0,T]$.  Then there exists a function $A: [0,T]\rightarrow\R^2$ such that, up to a subsequence,
$$ u(t,x) - \lim_{N\rightarrow -\infty} (H_N K) \ast \theta(t,x)= A(t),$$
where the limit holds in $L^{\infty}([0,T] \times K)$ for each compact set $K\subset \R^2$.
\end{lemma}
\begin{proof}
Let $\eta\in C^{\infty}_c([0,T]\times \R^2)$.  Since $(\theta,u)$ satisfies the homogeneous constitutive law at each $t\in [0,T]$, it follows that for each $(t,x)\in [0,T]\times\R^2$,
\begin{equation}\label{constlawapp}
\begin{split}
&(H_NK) \ast \theta(t,x) = \sum_{j= N}^{\infty} (\dot{\Delta}_jK)\ast \theta(t,x) = \sum_{j= N}^{\infty} \dot{\Delta}_j u(t,x) = u(t,x) - S_N u(t,x),
\end{split}
\end{equation}
so that $$u(t,x) - (H_NK) \ast \theta(t,x)  = S_Nu(t,x).$$
We claim that 
\begin{equation}\label{nablasnunif}
\|\eta\nabla S_Nu \|_{L^{\infty}([0,T]\times\R^2)}  \longrightarrow 0 \quad {\text  as } \quad N\rightarrow-\infty.
\end{equation}
Fix $N<0$.  Let $\psi\in C^{\infty}_c(\R^2)$ denote a smooth bump function with $\|\psi \|_{L^{\infty}} = 1$ satisfying 
$$\psi \equiv 1 \text{ on } \cup\{ \supp \eta(t,x): t\in [0,T]\}\subseteq\R^2. $$  By Minkowski's inequality for integrals,
\begin{equation}\label{nablasntozero}
\begin{split}
&\|\eta(t,x)\nabla S_Nu (t, x) \|_{L^2_x}  = \left \|  \int_{\R^2} \eta(t,x) \nabla \chi_N(y) u(t,x-y) \, dy   \right \|_{L^2_x}\\
&\qquad \leq  \int_{\R^2} \| \eta(t,x) \nabla \chi_N(y) u(t,x-y) \|_{L^2_x} \, dy  \\
&\qquad = \int_{\R^2} \abs{\nabla\chi_N(y)} \| u(t, x-y)\eta(t, x) \|_{L^2_x} \, dy\\
&\qquad \leq \| \nabla\chi_N\|_{L^{1}} \sup_{y\in\R^2}\| u(t, x-y)\psi(x) \|_{L^2_x}\\ 
&\qquad \leq C2^N\| u(t) \|_{L^2_{ul}} \longrightarrow 0
\end{split}
\end{equation}
as $N$ approaches negative infinity, where convergence is uniform over all $t\in [0,T]$.  By a similar argument using the product rule, one can show that for all multi-indices $\alpha$ with $|\alpha| \geq 1$,
\begin{equation}\label{dernablasnuHmconv} \sup_{t\in [0,T]}\|D^{\alpha}(\eta(t,x)\nabla S_Nu (t, x) )\|_{L^2_x}  \leq \sup_{t\in [0,T]} C2^N\| u(t) \|_{L^2_{ul}} \longrightarrow 0. 
\end{equation}
Therefore, for every $m\geq 0$,
\begin{equation}\label{nablasnuHmconv}\|\eta(t,x)\nabla S_Nu (t, x) \|_{L^{\infty}([0,T];H^m_x)}  \longrightarrow 0.
\end{equation}
By the Sobolev Embedding Theorem, \cref{nablasnunif} follows.  

Now note that by an argument identical to \cref{nablasntozero},
$$\|\eta  S_Nu  \|_{L^{\infty}([0,T]; L^2_x)} \leq C\|u \|_{L^{\infty}([0,T];L^2_{ul})}.$$  Moreover, by \cref{dernablasnuHmconv} and the product rule, we have 
\begin{equation}\label{etasnuisbounded}
\|\eta  S_Nu  \|_{L^{\infty}([0,T]; H^m_x)} \leq C\|u \|_{L^{\infty}([0,T];L^2_{ul})}
\end{equation}
for each $m\geq 0$.  Thus, by the Sobolev Embedding Theorem, $(\eta  S_Nu)$ is uniformly bounded with respect to $N<0$ on $[0,T]\times \R^2$, which implies that there exists a point $(t_0,x_0)$ and a subsequence of $(\eta(t_0,x_0)S_Nu(t_0,x_0))$ which converges.  This subsequence, defined on all of $[0,T]\times\R^2$, must converge in $L^{\infty}([0,T]\times \R^2)$ to a function whose spatial gradient is $0$, completing the proof. 
\end{proof}

\section{The non-uniqueness of solutions}\label{S:NonUniquess}

\noindent With lack of spatial decay, we cannot have uniqueness of solutions to \cref{e:SQGnu} or \cref{e:SQG0} without an additional uniqueness criterion. We show why this is the case in \cref{P:SQGnuNonunique,P:SQG0Nonunique}, adding an important piece of the puzzle later in \cref{P:SerfatiAway}.

In establishing this non-uniqueness, we will make the change of variables,
\begin{align}\label{e:COVx}
	\ox &= \ox(t, x) = x + \int_0^t U_\iny(s) \, ds
\end{align}
and set
\begin{align}\label{e:SolTranslation}
	\begin{split}
	\otheta(t, x)
		&= \theta(t, \ox), \\
	\ol{u}(t, x)
		&= u(t, \ox) - U_\iny(t).
	\end{split}
\end{align}
Note that \cref{e:COVx} is a Galilean transformation when $U_\iny$ is constant in time. These transformations are the same as those used in a periodic domain to allow the assumption that the average flow vanishes for the Navier-Stokes or Euler equations (see pages 31-32 of \cite{FMRT}.) 
 
In the context of bounded data solutions in the full space, the transformations in \cref{e:COVx,e:SolTranslation} yield distinct solutions to \cref{e:SQGnu}, $\nu \ge 0$, having the same initial data, demonstrating a type of non-uniqueness of solutions.
This was shown by Kukavica for the Navier-Stokes equations in \cite{KukavicaBoundedNS} and later for solutions to the 2D Euler equations in \cite{KBounded}.

\begin{prop}\label{P:SQGnuNonunique}
Assume that $\theta \in C^1(0, T; C^2(\R^2))$, $u$ is a divergence-free vector field in $C(0, T; L^2_{loc}(\R^2))$, and $(\theta,u)$ satisfy \cref{e:SQGnu}, $\nu \ge 0$. Assuming that $U_\iny \in C([0, T])$, the transformations in \cref{e:COVx,e:SolTranslation} yield another solution $(\ol{\theta}, \ou)$ to \cref{e:SQGnu}, which, if $U_\iny(0) = 0$, has the same initial data.
\end{prop}
\begin{proof}
By \cref{L:ConstitutiveTranslated}, $(\ol{\theta}, \ou)$ satisfies the homogeneous constitutive law. Using that
\begin{align*}
	\prt_t \ox
		&= U_\iny(t), \\
	\grad \ox
		&= \grad x
		= \Id,
\end{align*}
applying the chain rule gives
\begin{align*}
	\prt_t \otheta(t, x)
		&= \prt_t \theta(t, \ox) + \grad \theta(t, \ox) \cdot\prt_t \ox
		= \prt_t \theta(t, \ox) + U_\iny(t)
			\cdot \grad \theta(t, \ox), \\
	\grad \otheta(t, x)
		&= \grad \theta(t, \ox), \\
	\Delta \otheta(t, x)
		&= \Delta \theta(t, \ox),
\end{align*}
where here and below $\grad$ and $\Delta$ are with respect to $x$.
Hence,
\begin{align*}
	\prt_t \otheta(t, &x)
		+ \ou(t, x) \cdot \grad \otheta(t, x) 
			+ \nu \Delta \otheta(t, x) \\
		&= \prt_t \theta(t, \ox) + U_\iny(t)
			\cdot \grad \theta(t, \ox)
			+ (u(t, \ox) - U_\iny(t)) \cdot \grad \theta(t, \ox) \\
			&\qquad\qquad
			+ \nu \Delta \theta(t, \ox) \\
		&= \prt_t \theta(t, \ox)
			+ u(t, \ox) \cdot \grad \theta(t, \ox)
			+ \nu \Delta \theta(t, \ox)
		= 0.
\end{align*}
Thus, $(\otheta(t, x),\ou(t, x))$ is also a solution to $(SQG_{\nu})$. Moreover, if $U_\iny(0) = 0$, then, by \cref{e:COVx} and \cref{e:SolTranslation}, $(\bar{\theta}_0, \bar{u}_0) = (\theta_0, u_0)$.
\end{proof}

We used \cref{L:ConstitutiveTranslated} in the proof of \cref{P:SQGnuNonunique}, above.
\begin{lemma}\label{L:ConstitutiveTranslated}
	With the change of variables in \cref{e:COVx} and
	$\ol{\theta}$, $\ou$ as in \cref{e:SolTranslation},
	for any $j \in \Z$,
	$(\Del_j K) * \ol{\theta} = \Del_j \ou$ if and only if
	$(\Del_j K) * \theta = \Del_j u$.
\end{lemma}
\begin{proof}
Let us suppose that
$(\Del_j K) * \ol{\theta} = \Del_j \ou$.
Then, with $\al := \int_0^t U_\iny(s) \, ds$ (see \cref{S:LPOperators} for the precise definition of $\Del_j$),
\begin{align*}
	&\FTF \pr{(\Del_j K) * \ol{\theta}}(\xi)
		= \wh{\varphi_j}(\xi) \wh{K}(\xi)
			\wh{\ol{\theta}}(\xi)
		= \wh{\varphi_j}(\xi) \wh{K}(\xi)
			\FTF \pr{\theta(\cdot + \al)}(\xi) \\
		&\qquad
		= e^{i \al \cdot \xi} \wh{\varphi_j}(\xi) \wh{K}(\xi)
				\wh{\theta}(\xi)
		= e^{i \al \cdot \xi} \FTF \pr{(\Del_j K) * \theta)}(\xi),
\end{align*}
while, using that $\Del_j U_\iny(t) = 0$,
\begin{align*}
	&\FTF \pr{\Del_j \ol{u}}(\xi)
		= \FTF \pr{\Del_j (u(\cdot + \al) - U_\iny(t))}(\xi)
		= \FTF \pr{\Del_j (u(\cdot + \al))}(\xi) \\
		&\qquad
		= e^{i \al \cdot \xi}
			\wh{\varphi_j} (\xi) \wh{u}(\xi) 
		= e^{i \al \cdot \xi} \FTF \pr{\Del_j u}(\xi).
\end{align*}
But $(\Del_j K) * \ol{\theta} = \Del_j \ou$, so equating the expressions for their Fourier transforms above, we see that $(\Del_j K) * \theta = \Del_j u$. The proof of the converse implication proceeds in the same manner.
\end{proof}

\begin{prop}\label{P:SQG0Nonunique}
Let $(\theta, u)$ be a weak solution to \cref{e:SQG0} on $[0,T] \times \R^2$ as in \cref{weaksolutiondef2}. Assuming that $U_\iny \in C([0, T])$, the transformations in \cref{e:COVx,e:SolTranslation} yield another weak solution, $(\otheta, \ou)$, which, if $U_\iny(0) = 0$, has the same initial data.
\end{prop}
\newcommand{\ophi}{\ol{\phi}}
\begin{proof}
    Defining $z(t) := \int_0^t U_\iny(s) \, ds$, we can write
    \begin{align*}
        \ox = \ox(t, x) = x + z(t), \,
        \otheta(t, x) = \theta(t, \ox), \,
        \ou(t, x) = u(t, \ox) - U_\iny(t),
    \end{align*}
    so,
    \begin{align*}
        \prt_t \phi(t, x)
            &= \prt_t \phi(t, \ox - z(t))
            = (\prt_t \phi)(t, \ox - z(t))
            - (\grad \phi)(t, \ox - z(t)) \cdot U_\iny(t), \\
        \grad \phi(t, x)
            &= \grad_x \phi(t, \ox - z(t))
            = \grad_{\ox} \phi(t, \ox - z(t)).
    \end{align*}
    We can write this as
    \begin{align*}
        \prt_t \phi(t, x)
            &= \prt_t \ophi(t, \ox) - \grad \ophi(t, \ox) \cdot U_\iny(t), \\\
        \grad \phi(t, x)
            &= \grad \ophi(t, \ox), \\
        \ophi(t, y)
            &:= \phi(t, y - z(t)),
    \end{align*}
    and hence $\ophi \in C^1_c([0, T) \times \R^2)$.
    Using these expressions, we have
    \begin{align*}
        \int_0^T &\int_{\R^2} (\prt_t \phi \, \otheta)(t, x) \, dx \, dt
                - \int_0^T \int_{\R^2}
                    (\grad \phi \cdot \ou \, \otheta)(t, x) \, dx \, dt \\
            &= \int_0^T \int_{\R^2}
                \pr{\prt_t \ophi(t, \ox) - \grad \ophi(t, \ox) \cdot U_\iny(t)}
                \theta(t, \ox)
                \, d \ox \, dt \\
            &\qquad
                - \int_0^T \int_{\R^2}
                    \grad \ophi(t, \ox) \cdot
                    (u(t, \ox) - U_\iny(t)) \, \theta(t, \ox)
                \, d \ox \, dt \\
           &= \int_0^T \int_{\R^2}
                \prt_t \ophi(t, \ox)
                \theta(t, \ox)
                \, d \ox \, dt
                - \int_0^T \int_{\R^2}
                    \grad \ophi(t, \ox) \cdot
                    u(t, \ox) \, \theta(t, \ox)
                \, d \ox \, dt \\
            &= \int_{\R^2}\ophi(T)\theta(T)-\int_{\R^2} \ophi(0) \theta_0
            = \int_{\R^2}\ophi(T)\otheta(T)-\int_{\R^2} \ophi(0) \otheta_0,
    \end{align*}
    the second-to-last equality holding because $(\theta, u)$ is a weak solution with the same initial data as $(\otheta, \ou)$. This shows, in light of \cref{R:WeakerTestFunctions}, that $(\otheta, \ou)$ is a weak solution.
\end{proof}

The non-uniqueness of \cref{P:SQG0Nonunique} arises from allowing non-decaying data; left open is whether non-uniqueness still holds if we assume the initial data decays rapidly at infinity.

\section{A Serfati identity and a priori estimates}\label{S:Apriori}
In this section, we establish a Littlewood-Paley version of a Serfati identity satisfied by smooth non-decaying solutions to ($SQG_{\nu}$) with $\nu>0$.  We utilize this identity to prove the necessary a priori bounds on smooth, non-decaying solutions, as given by \cref{T:SQG1}.  In \cref{S:Existence}, we will apply these bounds to an approximation sequence of smooth solutions and pass to the limit as $\nu\rightarrow 0$ to yield a weak solution to ($SQG_0$).

We begin with a statement of the Serfati identity and an a priori estimate on the $L^2_{ul}$-norm of the velocity. 

\begin{theorem}\label{theoremapriori}
Let $k\geq 2$, $\nu>0$, $T>0$, $N \in \Z$, and let $(\theta,u)$ be a smooth, non-decaying solution to ($SQG_{\nu}$) on $[0,T]$ with initial data $u_0,\theta_0\in C_b^k(\R^2)$.  Then there exists $U_{\infty}\in C([0,T])$ independent of $N$, with $U_{\infty}(0)=0$, such that for all $t\in [0,T]$,  
\begin{equation}\label{FourierSerfati2}
\begin{split}
u(t,x) - u_0(x)  &= U_{\infty}(t) -\int_0^t \left( S_{N}\nabla K \stardot (\theta u) -  \nu S_{N}\Delta K \ast \theta\right)(s,x) \, ds\\
&\qquad\qquad + (H_NK) \ast (\theta(t,x) - \theta_0(x)).
\end{split}
\end{equation}
Moreover, the following estimate holds for each $t\in [0,T]$:
\begin{equation}\label{L2ulestDSQG}
\| u(t) \|_{L^2_{ul}} \leq C\left(\| u_0 \|_{L^2_{ul}} + \|\theta_0 \|_{L^{\infty}} +|U_{\infty}(t)|+\nu\right)\exp \left( Ct\|\theta_0\|_{L^{\infty}}\right). 
\end{equation}  
\end{theorem}

Before proving \cref{theoremapriori}, we establish the following lemma, which estimates the low frequencies in \cref{FourierSerfati2}.
\begin{lemma}\label{lowfreqwithN}
Let $\theta,u$ be as in \cref{theoremapriori} above.  For any $N\in\Z$, $(t,x)\in [0,T]\times\R^2$, and $\eta\in C^{\infty}_c(\R^2)$, if we set 
\begin{equation*}
l_N(t,x) = -\eta(x)\int_0^t \left( S_{N}\nabla K \stardot (\theta u)  - \nu S_{N}\Delta K \ast \theta\right)(s,x) \, ds,
\end{equation*}
then 
\begin{equation*} 
\begin{split}
&\| l_N(t,x) \|_{L^2_x} \leq C(2^{N}+ 2^{2N})\| \theta_0 \|_{L^{\infty}}\left(t\nu + \int_0^t \| u(s) \|_{L^2_{ul}}\, ds \right).
\end{split}
\end{equation*}
\end{lemma}

\begin{proof}
First write $$l_N(t,x) =l_N^1(t,x) +l_N^2(t,x),$$ where 
\begin{equation*}
\begin{split}
&l_N^1(t,x) := -\eta(x)\int_0^t S_{N}\nabla K \stardot (\theta u) (s,x) \, ds,\\
&l_N^2(t,x) :=  \nu\eta(x)\int_0^t S_{N}\Delta K \ast \theta(s,x) \, ds. 
\end{split}
\end{equation*}  

Note that, by Minkowski's inequality for integrals,
\begin{equation}\label{I^1stepone}
\begin{split}
&\|l_N^1(t, x) \|_{L^2_x} = \left \| \eta(x)\int_0^t S_{N}\nabla K \stardot (\theta u)(s,x)  \, ds \right \|_{L^2_x}\\
&\qquad = \left \| \eta(x)\int_0^t \int_{\R^2}S_{N}\nabla K(y) \cdot(\theta u)(s,x-y) \, dy  \, ds \right \|_{L^2_x}\\
&\qquad \leq  \int_0^t\int_{\R^2} \| \eta(x) S_{N}\nabla K(y) \cdot(\theta u)(s,x-y) \|_{L^2_x} \, dy  \, ds \\
&\qquad = \int_0^t\int_{\R^2} \abs{\nabla S_{N}K(y)} \| (\theta u)(s, x-y)\eta(x) \|_{L^2_x} \, dy\, ds\\
&\qquad \leq \int_0^t\| \theta(s)\|_{L^{\infty}} \sup_{y\in\R^2}\| u(s, x-y)\eta(x) \|_{L^2_x}\, ds \int_{\R^2} \abs{\nabla S_{N}K(y)}  \, dy\\
&\qquad \leq \int_0^t\| \theta(s)\|_{L^{\infty}} \| u(s) \|_{L^2_{ul}}\, ds \int_{\R^2} \abs{\nabla S_{N}K(y)}  \, dy.
\end{split}
\end{equation}
By \cref{SNKbound}, $\| \nabla S_{N}K \|_{L^1} \leq C2^{N}$.  It then follows from \cref{I^1stepone} and \cref{max} that \begin{equation}\label{lowfreqjust}\|l_N^1(t,x) \|_{L^2_x} \leq C2^{N}\int_0^t\| \theta(s)\|_{L^{\infty}} \| u(s) \|_{L^2_{ul}}\, ds \leq C2^{N}\int_0^t\| \theta_0\|_{L^{\infty}} \| u(s) \|_{L^2_{ul}}\, ds.
\end{equation}
To estimate $l_N^2$, we observe that 
\begin{equation}\label{ln:l1-l2-bound}
\begin{split}
&\|l_N^2(t, x) \|_{L^2_x} \leq \nu\| \eta\|_{L^2} \left \|\int_0^t (S_{N}\Delta K \ast \theta)(s,x)  \, ds \right \|_{L^{\infty}_x}\\
&\qquad \leq \nu\| \eta\|_{L^2}\int_0^t\| \theta(s) \|_{L^{\infty}}\, ds \int_{\R^2} \abs{\Delta S_{N}K(y)}  \, dy.
\end{split}
\end{equation}
Again by \cref{SNKbound},
\begin{equation*}
    \|S_N \Delta K\|_{L^1} \leq C2^{2N}.
\end{equation*}
Substituting this estimate into \eqref{ln:l1-l2-bound} and applying \cref{max}, we obtain
\begin{equation} \label{ln:bound-on-low-freq-no-deriv}
\|l_N^2(t, x) \|_{L^2_x} \leq Ct2^{2N}\nu \| \theta_0 \|_{L^{\infty}}.
\end{equation}
Combining the estimates for $l_N^1$ and $l_N^2$ gives
\begin{equation*}
\begin{split}
&\| l_N(t,x) \|_{L^2_x} \leq \|l_N^1 (t,x) \|_{L^2_x} + \|l_N^2 (t,x) \|_{L^2_x}\\
&\qquad  \leq C(2^{N}+ 2^{2N})\| \theta_0 \|_{L^{\infty}}\left(t\nu + \int_0^t \| u(s) \|_{L^2_{ul}}\, ds \right).
\qedhere
\end{split}
\end{equation*}
\end{proof}
\begin{proof}[\textbf{Proof of \cref{theoremapriori}}] 
We first establish \eqref{FourierSerfati2} for a fixed $N \in \Z$.  We begin with $(SQG_{\nu})_1$,
\begin{equation}
\frac{\partial\theta}{\partial t}+u\cdot\nabla\theta 
=\nu\,\Delta\theta.
\end{equation}
We integrate in time and use the divergence-free property of $u$ to get 
$$\theta(t,x) - \theta_0(x) = -\int_0^t( \dv (\theta u) - \nu\Delta\theta)(s,x) \, ds.$$
The above equality, the fact that $(\theta,u)$ satisfy the homogeneous constitutive law, and integration by parts give
\begin{equation*}
\begin{split}
&\dot{\Delta}_j \left(u(t,x) - u_0(x)\right) = S_{N}\left((\dot{\Delta}_jK) \ast (\theta(t) - \theta_0)(x)\right) + H_N \left((\dot{\Delta}_jK) \ast (\theta(t) - \theta_0)(x)\right) \\
&\qquad = - S_{N}\left((\dot{\Delta}_jK) \ast \int_0^t(\dv (\theta u) -\nu\Delta\theta)(s,x)  \, ds \right)+ H_N \left((\dot{\Delta}_jK) \ast (\theta(t) - \theta_0)(x) \right)\\
& = \dot{\Delta}_j\left(-\int_0^t \left( S_{N}\nabla K \stardot (\theta u)  - \nu S_{N}\Delta K \ast \theta\right)(s,x) \, ds +  (H_NK) \ast (\theta(t) - \theta_0)(x) \right).
\end{split}
\end{equation*}
As the above inequality holds for all $j\in\Z$, we conclude that 
\begin{equation*}
\begin{split}
&u(t,x) - u_0(x) = -\int_0^t \left( S_{N}\nabla K \stardot (\theta u)  - \nu S_{N}\Delta K \ast \theta\right)(s,x) \, ds\\
&\qquad +  (H_NK) \ast (\theta(t) - \theta_0)(x) +P(t,x),
\end{split}
\end{equation*}
where for each fixed $t\in[0,T]$, $P(t,x)$ is a polynomial in the spatial variable.  Now note that by \cref{SNKbound}, $S_N\nabla K$ and $S_N\Delta K$ belong to $L^1(\R^2)$. Since $\theta, u$ are in $L^{\infty}([0,T]; L^{\infty}(\R^2))$ by \cref{T:SQG1}, we can apply Young's inequality and conclude that the time integral above lies in $L^{\infty}([0,T]\times\R^2)$.  Moreover, by the argument in \cref{constlawapp}, 
$$u(t,x) - u_0(x) -  (H_NK) \ast (\theta(t) - \theta_0)(x) = S_N (u(t) - u_0)(x)\in L^{\infty}([0,T]\times\R^2),$$
again by \cref{T:SQG1}.  We conclude that $P(t,x) \in L^{\infty}(\R^2)$ for each fixed $t\in [0,T]$, implying that $P(t,x) = U_{\infty}(t)$, where $U_{\infty}$ is constant in $x$.  Continuity of $U_{\infty}$ in time again follows from regularity properties of $u$ and $\theta$ given by \cref{T:SQG1}.  This gives \cref{FourierSerfati2}.

To show that $U_{\infty}$ is independent of $N$, subtracting the second and third terms on the right side of \cref{FourierSerfati2} for $N=N_1$ from the corresponding terms for $N=N_2$ with $N_1 < N_2$ gives
\begin{align*}
    \sum_{j = N_1 + 1}^{N_2}
            &\brac{
                -\int_0^t
                ((\Del_j \grad K) \stardot (\theta u)(s) - \nu (\Del_j\Delta K) \ast \theta(s) ) \, ds
                    - (\Del_j K) * (\theta(t) - \theta_0)
            } \\
        &= -\sum_{j = N_1 + 1}^{N_2}
                \int_0^t (\Del_j K) * (\dv(\theta u)(s) - \nu\Delta\theta(s) + \prt_s \theta(s)) \, ds
        = 0,
\end{align*}
from which it follows that $U_{\infty}$ does not depend on $N$.

We now show \cref{L2ulestDSQG}. Let $\eta\in H_R(\R^2)$ for some $R>0$, and multiply \cref{FourierSerfati2} through by $\eta$, which gives
\begin{equation}\label{prelim}
\begin{split}
&\eta(x) u(t,x) - \eta(x) u_0(x) = -\eta(x)\int_0^t \left( S_{N}\nabla K \stardot (\theta u)  - \nu S_{N}\Delta K \ast \theta\right)(s,x) \, ds\\
&\qquad\qquad  + \eta(x)((H_NK) \ast (\theta(t) - \theta_0))(x) + \eta(x)U_{\infty}(t).
\end{split}
\end{equation} 
For ease of notation, for each $N\in\Z$, set the low frequency term and high frequency term as follows:
\begin{equation} \label{ln:def-of-I-and-II}
\begin{split}
&l_N(t,x) = -\eta(x)\int_0^t \left( S_{N}\nabla K \stardot (\theta u)  - \nu S_{N}\Delta K \ast \theta\right)(s,x) \, ds,\\
&h_N(t,x) = \eta(x)((H_NK) \ast (\theta(t) - \theta_0))(x).
\end{split}
\end{equation}
By \cref{lowfreqwithN} with $N=0$,
$$ \| l_0(t,x) \|_{L^2_x} \leq C_1\| \theta_0 \|_{L^{\infty}}\left(t\nu + \int_0^t \| u(s) \|_{L^2_{ul}}\, ds \right).$$

To estimate $h_0$, let $\psi\in C^{\infty}_c(\R^2)$ be equal to one on the support of $\eta$ and observe that
\begin{equation}\label{DefOfIandII}
\begin{split}
&\psi(x)\eta(x)((H_0K) \ast \theta(t,x) ) = \psi(x)\eta(x)\int_{\R^2} (H_0K)(x-y) \theta(t,y) \, dy\\
&\qquad = \eta(x)[\psi, (H_0K)\ast]\theta(x) + \eta(x)\int_{\R^2} (H_0K)(x-y) \psi(y) \theta(t,y) \, dy\\
& \qquad := h^1(t,x) + h^2(t,x).
\end{split}
\end{equation}
Since $K$ is a Calderon-Zygmund kernel and therefore bounded on $L^2(\R^2)$, the second term $h^2$ is estimated as follows:
\begin{equation*}
\begin{split}
\| h^2(t,x) \|_{L^2_x} &\leq C\|(Id-S_0)K\ast (\psi\theta(t))\|_{L^2}\\
&\leq C\|K\ast (\psi\theta(t))\|_{L^2} + C\|S_0(K\ast (\psi\theta(t)))\|_{L^2}\\
&\leq C\| \psi\theta(t) \|_{L^{2}} \leq C\| \theta(t) \|_{L^{\infty}} \leq C\| \theta_0 \|_{L^{\infty}}.  
\end{split}
\end{equation*}
For $h^1$, first write
\begin{equation}\label{H0stepone}
\begin{split}
&\| h^1(t) \|_{L^2_x} \leq \| \eta \|_{L^2} \|[\psi, (H_0 K)\ast]\theta\|_{L^{\infty}}\\
&\qquad \leq \| \eta \|_{L^2} \sum_{j\geq 0}\|[\psi, (\Del_j K)\ast]\theta\|_{L^{\infty}}.
\end{split}
\end{equation}
An application of \cref{LPRieszcommutator} and \cref{max} then gives
$$ \| h^1(t,x) \|_{L^2_x} \leq C\| \theta_0\|_{L^{\infty}}.$$
Combining the estimates for $h^1$ and $h^2$ gives
\begin{equation}\label{highfreq}
\| \eta(x)((H_0K) \ast \theta(t,x)) \|_{L^2_x} \leq C\|\theta_0 \|_{L^{\infty}}.
\end{equation}  
Similarly, we have
\begin{equation*}
\| \eta(x)((H_0K) \ast \theta_0(x)) \|_{L^2_x} \leq C\|\theta_0 \|_{L^{\infty}}.
\end{equation*}  
Thus, from the definition of $h_0(t,x)$, it follows easily that  
\begin{equation}\label{h0}
\|h_0(t,x) \|_{L^2_x} \leq C_2\|\theta_0 \|_{L^{\infty}}.
\end{equation}
Finally, substituting the above estimates for $l_0$ and $h_0$ into \cref{prelim} and taking the supremum over $\eta\in H_R(\R^2)$, we have
\begin{equation}
    \| u\|_{L^2_{ul}} \leq \|u_0\|_{L^2_{ul}} + |U_{\infty}(t)| +C_2\|\theta_0\|_{L^\infty} +  C_1\| \theta_0 \|_{L^{\infty}}\left(t\nu + \int_0^t \| u(s) \|_{L^2_{ul}}\, ds \right). 
\end{equation}
We are now in a position to use a linear Gronwall inequality (see Theorem 6 of \cite{dragomir2003some}) to give the final result,
\begin{equation}
    \|u\|_{L^2_{ul}} \leq \nu \exp(C_1\|\theta_0\|_{L^\infty}t -1) + \left(\| u_0\|_{L^2_{ul}} + |U_{\infty}(t)|+C_2\|\theta_0\|_{L^\infty} \right)\exp(C_1\|\theta_0\|_{L^\infty} t).
\end{equation}
\end{proof}
The next lemma, giving smoothness of the low frequency term in \cref{FourierSerfati2}, will allow us to conclude that the low frequency term vanishes uniformly on $[0,T]\times \R^2$ as $N$ approaches negative infinity. 
\begin{lemma}\label{IH12}
Let $m\geq 1$ be an integer.  Let $(\theta,u)$ be a smooth solution to $(SQG_{\nu})$ on $[0,T]$, given by \cref{T:SQG1}. Define $l_N(t,x)$ as in \cref{ln:def-of-I-and-II} with $N<0$. That is,
\begin{align*}
    l_N(t,x) =  -\eta(x)\int_0^t \left( S_{N}\nabla K \stardot (\theta u) - \nu S_{N}\Delta K \ast \theta\right)(s,x) \, ds.
\end{align*}
Then for $t\in[0,T]$,
\begin{equation}
\| l_N(t,x) \|_{H^m_x} \leq C2^{N}\| \theta_{0}\|_{L^{\infty}}\left(t\nu + \int_0^t \| u \|_{L^2_{ul}}\, ds\right).
\end{equation}
\end{lemma}
\begin{proof}
By \cref{lowfreqwithN} with $N<0$,
\begin{equation}\label{IL2}
\| l_N(t,x) \|_{L^2_x} \leq C2^{N}\| \theta_{0}\|_{L^{\infty}}\left(t\nu + \int_0^t \| u \|_{L^2_{ul}}\, ds\right).
\end{equation}
To estimate first order derivatives of $l_N$, as in the proof of \cref{lowfreqwithN}, let $$l(t,x) =l_N^1(t,x) +l_N^2(t,x),$$ where 
\begin{equation*}
\begin{split}
&l_N^1(t,x) := -\eta(x)\int_0^t S_{N}\nabla K \stardot (\theta u) (s,x) \, ds,\\
&l_N^2(t,x) :=  \nu\eta(x)\int_0^t S_{N}\Delta K \ast \theta(s,x) \, ds, 
\end{split}
\end{equation*} 
and note that for $1\leq j \leq 2$,  
\begin{equation*}
\begin{split}
\partial_j l_N^1(t,x)  &=  \partial_j\left (\eta(x)\int_0^t (S_{N}\nabla K \stardot (\theta u))(s,x)  \, ds \right)\\
&= \underbrace{ \partial_j\eta(x)\int_0^t (S_{N}\nabla K \stardot (\theta u))(s,x)  \, ds}_{=: i(t)}  + \underbrace{ \eta(x)\int_0^t ((\partial_jS_{N}\nabla K) \stardot (\theta u))(s,x)  \, ds }_{=: ii(t)}.
\end{split}
\end{equation*}
Thus,
$$\| \partial_j l_N^1(t,x)\|_{L^2_x}\leq \|i(t)\|_{L^2_x} + \|ii(t)\|_{L^2_x}.$$
We note, $$\|i(t)\|_{L^2_x}\leq C2^{N}\int_0^t\| \theta_0\|_{L^{\infty}} \| u \|_{L^2_{ul}}\, ds$$ by an argument identical to that used to bound $\|l_N^1 (t)\|_{L^2_{ul}}$ in the proof of \cref{lowfreqwithN}.  For $ii$, we apply \cref{SNKbound} and, again, and argument similar to that used to bound $\|l_N^1 (t)\|_{L^2_{ul}}$ in the proof of \cref{lowfreqwithN}, which yields $$\|ii(t)\|_{L^2_x}\leq C2^{2N}\int_0^t\| \theta_0\|_{L^{\infty}} \| u \|_{L^2_{ul}}\, ds.$$   

As $N<0$, we conclude that 
$$\|l_N^1(t,x) \|_{H^1_x} \leq C2^{N}\int_0^t\| \theta_0\|_{L^{\infty}} \| u(s) \|_{L^2_{ul}}\, ds.$$ 
By a similar argument,
\begin{equation*}
\begin{split}
&\| \partial_j l_N^2(t,x) \|_{L^2_x} = \nu\left \|  \partial_j\left (\eta(x)\int_0^t S_{N}\Delta K \ast \theta)(s,x)  \, ds\right) \right \|_{L^2_x} \\
&\leq  \nu\left\|\partial_j\eta(x)\int_0^t S_{N}\Delta K \ast \theta (s,x)  \, ds \right\|_{L^2_x} + \nu\left\|\eta(x)\int_0^t (\partial_j S_{N}\Delta K) \ast \theta (s,x)  \, ds \right\|_{L^2_x}.
\end{split}
\end{equation*}
We apply the estimate \eqref{ln:bound-on-low-freq-no-deriv} from \cref{lowfreqwithN} to bound the first term.  This gives
$$\nu\left\|\partial_j\eta(x)\int_0^t S_{N}\Delta K \ast \theta (s,x)  \, ds \right\|_{L^2_x} \leq Ct2^{2N}\nu \|\theta_0\|_{L^{\infty}}.$$
For the second term, utilizing the compact support of $\eta$ and Young's inequality, we have
\begin{equation*}
\begin{split}
\left\|\eta(x)\int_0^t (\partial_j S_{N}\Delta K) \ast \theta (s,x)  \, ds \right\|_{L^2_x} \leq C\|\eta\|_{L^{2}}\|\theta_0\|_{L^{\infty}} \int_0^t\left\| (\partial_j S_{N}\Delta K) (x) \right\|_{L^{1}_x} \, ds.
\end{split}
\end{equation*}
By \cref{SNKbound},
we conclude that
\begin{equation}
   \nu\left\|\eta(x)\int_0^t (\partial_j S_{N}\Delta K) \ast \theta (s,x)  \, ds \right\|_{L^2_x} \leq 2^{3N}Ct\nu \|\theta_0\|_{L^{\infty}}.
\end{equation}
Combining the estimates for the first and second term and noting that $N<0$, we have
\begin{equation*}
\begin{split}
&\| \partial_j l_N^2(t,x) \|_{L^2_x} \leq 2^{N}Ct\nu \|\theta_0\|_{L^{\infty}}.
\end{split}
\end{equation*}
This completes the proof of \cref{IH12} for the case $m=1$.  The lemma can be extended to general $m$ by induction.
\end{proof}

The next theorem gives an alternative to the Serfati identity. 
\begin{prop}\label{renormalized2}
Let $(\theta,u)$ be a smooth solution to ($SQG_{\nu}$) on $[0,T]$, given by \cref{T:SQG1}.  Then 
\begin{equation}\label{highfreq0limit}
\begin{split}
&u(t) - u_0 -U_{\infty}(t) = \lim_{N\rightarrow -\infty} (H_NK) \ast (\theta(t) - \theta_0)
\end{split}
\end{equation}
uniformly on compact subsets of $[0,T]\times \R^2$.
\end{prop}

\begin{proof}
Multiplying \cref{FourierSerfati2} by $\eta\in C^{\infty}_c([0,T]\times\R^2)$ gives
\begin{equation*}
\begin{split}
&\eta(t,x)(u(t,x) - u_0(x) -U_{\infty}(t)) = \eta(t,x) \int_0^t \left( S_{N}\nabla K \stardot (\theta u) - \nu S_{N}\Delta K \ast \theta\right)(s,x) \, ds  \\
&\qquad + \eta(t,x) ( (H_NK) \ast (\theta(t,x) - \theta_0(x))).  \\
\end{split}
\end{equation*}
By \cref{IH12}, \cref{L2ulestDSQG}, and the Sobolev Embedding Theorem, 
\begin{equation}\label{lowfreqto0uniform}\left \| \eta(t,x) \int_0^t \left( S_{N}\nabla K \stardot (\theta u) - \nu S_{N}\Delta K \ast \theta\right)(s,x) \, ds \right \|_{L^{\infty}([0,T]\times\R^2)} \rightarrow 0
\end{equation} 
as $N$ approaches negative infinity, giving \cref{highfreq0limit}.
\end{proof}

\begin{prop}\label{P:SerfatiAway}
Let $(\theta, u)$ be as in \cref{theoremapriori} and, assuming $U_\iny(0) = 0$, make the transformations in \cref{e:COVx,e:SolTranslation} to yield another solution $(\otheta, \ou)$ with initial data $(\theta_0, u_0 - U_\iny(0))$. If $(\theta,u)$ satisfy the Serfati identity in \cref{FourierSerfati2} for the vector field $V_\iny$, then $(\bar{\theta}, \ou )$ satisfy \cref{FourierSerfati2} in the form
\begin{align*}
	\ou(t) - u_0
		&= W_\iny(t)
				- \int_0^t \brac{(S_N \grad K) \stardot (\otheta \ou)
						- \nu (S_N \Delta K) * \otheta}
						(s) \, ds \\
		&\qquad\qquad
				+ (H_N K) * (\otheta(t) - \theta_0),
\end{align*}
where $W_\iny = V_\iny - U_\iny$.
\end{prop}
\begin{proof}
    By \cref{P:SQG0Nonunique,theoremapriori} we know that $( \otheta,\ou)$ satisfy the Serfati identity for some continuous-in time-vector field $W_\iny(t)$ that is independent of $N$. Because $U_\iny(0) = 0$, we have $\ou(t, x) = u(t, \ox)-U_{\infty}(t)$, $\otheta(t, x) = \theta(t, \ox)$, $\ou_0 = u_0$, $\otheta_0 = \theta_0$, and $\ox(0, x) = x$. Applying \cref{renormalized2} gives
    \begin{align*}
        \ou(t, &x) - \ou_0(x)
            = W_\iny(t) + \lim_{N \to -\iny}
                (H_N K) * (\otheta(t, x) - \otheta_0(x)) \\
        &\implies u(t, \ox) - U_\iny(t) - u_0(x)
            = W_\iny(t) + \lim_{N \to -\iny}
                (H_N K) * (\theta(t, \ox) - \theta_0(x)) \\
        &\implies W_\iny(t)
            = u(t, \ox) - \lim_{N \to -\iny} (H_N K) * (\theta(t, \ox)  \\
            &\qquad\qquad\qquad\qquad
                - [u_0(x) - \lim_{N \to - \iny} (H_N K) * \theta_0(x)]
                    - U_\iny(t)
                \\
            &\qquad\qquad\quad\;\;
            = V_\iny(t) - U_\iny(t).
            \qedhere
         \end{align*}
\end{proof}

The following is an immediate corollary of \cref{P:SQG0Nonunique,P:SerfatiAway}.

\begin{cor}\label{C:SerfatiAway}
    In \cref{theoremapriori}, we can choose the vector field $U_{\infty}$ to be
    any continuous-in-time vector field as long as $U_{\infty}(0) = 0$.
\end{cor}
\begin{remark}
    Each choice of $U_{\infty}$ in \cref{theoremapriori} will yield a distinct solution to \cref{e:SQGnu} having the same initial conditions.
\end{remark}

\section{Construction of candidate for a weak solution}\label{S:SolutionCandidate}

In this section, we assume our initial data ($\theta_{0},u_{0}$) belongs to $L^{\infty}(\R^2)\times L^2_{ul}(\R^2)$, satisfies the homogeneous constitutive law, and satisfies div $u_0=0$.
\subsection{Preparing the initial data}
To establish existence of weak solutions, we will generate a sequence of smooth solutions ($\theta_k, u_k$) to $(SQG_{1/k})$ from the initial data $$(\theta_{k,0},u_{k,0})=(S_k\theta_{0}, S_k u_{0}).$$  Note that by Young's inequality and \cref{SnBound},
\begin{equation} \label{ln:uniform-bound-on-initial-data-of-sequence}
\begin{split}
&\| \theta_{k,0} \|_{L^{\infty}} = \| S_k\theta_{0} \|_{L^{\infty}} \leq C \| \theta_0 \|_{L^{\infty}}, \\
&\| u_{k,0} \|_{L^2_{ul}} = \| S_ku_{0} \|_{L^2_{ul}} \leq C \| u_{0} \|_{L^2_{ul}}.
\end{split}
\end{equation}
As a consequence of \cref{T:SQG1}, for an arbitrary time $T>0$, there exists a sequence of smooth, non-decaying solutions $(\theta_k, u_k)$ to ($SQG_{1/k}$) with initial data $(\theta_{k,0}, u_{k,0})$.  
 Moreover, by \cref{max} and \cref{ln:uniform-bound-on-initial-data-of-sequence},
\begin{equation}\label{ln:uniform-bound-on-theta-k}
    \|\theta_k\|_{L^\infty([0,T]; L^{\infty})} \leq \|\theta_{k,0}\|_{L^\infty} \leq \|\theta_0\|_{L^\infty}.
\end{equation}

We now show that $u_{k,0}$ converges to $u_0$ locally in $L^2(\R^2)$.

\begin{prop}\label{prop:convergence-of-initial-data-u}
Let $\eta\in C^{\infty}_c(\R^2)$.  Then $$ \eta S_k u_0 \rightarrow  \eta u_0 \quad \text{in }L^2(\R^2) \text{ as }k\rightarrow\infty.$$
\end{prop}

\begin{proof}
Let $\psi$ be a smooth, nonnegative, compactly supported function which is identically $1$ on the support of $\eta$. Write 
\begin{equation*}
\| \eta S_k u_0 -  \eta u_0 \|_{L^{2}} \leq \|\psi [\eta, S_k] u_0 \|_{L^2} + \| S_k(\eta u_0) - \eta u_0 \|_{L^2}.
\end{equation*}
Since $\eta u_0$ belongs to $L^{2}(\R^2)$, the second term on the right hand side approaches $0$ as $k$ approaches $\infty$ by a standard approximation to the identity argument.  To see that the first term approaches $0$, we expand the commutator, which gives
\begin{equation*}
\begin{split}
&\abs{\psi(x)[\eta, S_k] u_0(x)} = \abs{\psi(x)2^{2k}\int_{\R^2} \chi(2^k(x-y))(\eta(x) - \eta(y)) u_0(y) \, dy}  \\
&\qquad \leq \abs{\psi(x)}\|\nabla \eta \|_{L^{\infty}} 2^{2k} \int_{\R^2} |\chi(2^k(x-y))||x-y| |u_0(y)| \, dy\\
&\qquad \leq \abs{\psi(x)}\|\nabla \eta \|_{L^{\infty}} 2^{-k} \int_{\R^2} |\chi(z)||z| |u_0(x-2^{-k}z)| \, dz.
\end{split}
\end{equation*}
Thus, by Minkowski's inequality for integrals,
\begin{equation*}
\begin{split}
&\| \psi[\eta, S_k] u_0\|_{L^2} \leq \|\nabla \eta \|_{L^{\infty}} 2^{-k} \int_{\R^2} |\chi(z)||z|\left(\int_{\R^2}|\psi(x)|^2 |u_0(x-2^{-k}z)|^2 \, dx \right)^{1/2} \, dz\\
&\qquad\qquad \leq C2^{-k} \sup_{z\in\R^2} \left(\int_{\R^2}|\psi(x)|^2 |u_0(x-2^{-k}z)|^2 \, dx \right)^{1/2} \leq C2^{-k} \|u_0\|_{L^2_{ul}}.
\end{split}
\end{equation*}
The desired convergence follows.
\end{proof}


\subsection{Passing to the limit of smooth solutions}
We now establish convergence of a sequence of smooth, non-decaying solutions ($\theta_k,u_k$) to ($SQG_{1/k}$), generated from initial data ($S_k\theta_0, S_ku_0$), to a candidate for a weak solution to ($SQG_0$).  

In what follows, to simplify calculations, for each $k$, given initial data $(S_k\theta_0, S_ku_0)$, we will choose the corresponding solution $(\theta_k, u_k)$ satisfying \cref{FourierSerfati2} with 
\begin{equation}\label{Uequals0}
U_{\infty}\equiv 0 \text{ on } [0,T].
\end{equation}
This choice is justified by \cref{C:SerfatiAway}. 

\begin{prop}\label{prop:existence-of-limit-of-theta-u}
     Suppose $(\theta_0,u_0) \in L^{\infty}(\R^2)\times L^2_{ul}(\R^2)$ satisfies the homogeneous constitutive law, with div $u_0=0$, and let $(\theta_k,u_k)$ be a solution to $(SQG_{1/k})$ on $[0,T]$ with initial data given by $(\theta_{k,0},u_{k,0})=(S_k\theta_{0}, S_k u_{0})$. Then there exists a pair $(\theta,u)$ such that, up to a subsequence, $$\theta_k\rightarrow\theta \text{ weak-$\star$ in }L^{\infty}([0,T]\times \R^2),$$ and 
     $$u_k \rightarrow u \text{ weak-$\star$ in }L^{\infty}([0,T]; L^2(K)) \text{ for each compact $K$ in $\R^2$}.$$ Further, $$\theta \in L^\infty([0,T]\times\R^2),$$ and $$u \in L^{\infty}([0,T];L^2_{ul}(\R^2)).$$
\end{prop}
\begin{proof}
Let $\eta\in C^{\infty}_c(\R^2)$ be a standard bump function, and, for each $R>0$, set $\eta_R (x)= \eta (x/R)$.  Note that with \eqref{L2ulestDSQG} and \eqref{ln:uniform-bound-on-initial-data-of-sequence}, for each $R>0$, $t\in [0,T]$, and $z\in \R^2$, we have the uniform bound
\begin{equation}\label{velunifbound}
\begin{split}
\| \eta_R(\cdot-z) u_k(t) \|_{L^2} &\leq CR\| u_k(t) \|_{L^2_{ul}} \\
&\leq CR(\| u_{k,0} \|_{L^2_{ul}} + \| \theta_{k,0} \|_{L^{\infty}} + 1/k)e^{Ct \| \theta_{k,0} \|_{L^{\infty}}}\\
&\leq CR(\| u_{0} \|_{L^2_{ul}} + \| \theta_{0} \|_{L^{\infty}} +1)e^{Ct \| \theta_{0} \|_{L^{\infty}}}.
\end{split}
\end{equation}
Restricting to $z=0$, it follows from a standard diagonalization argument that there exists $u$ belonging to $L^{\infty}([0,T]; L_{loc}^2(\R^2))$ such that, up to a subsequence, 
\begin{equation}\label{L2weakstarconv}
\begin{split}
&u_k \rightarrow u \text{ weak-$\star$ in }L^{\infty}([0,T]; L^2(K)) \text{ for each compact $K$ in $\R^2$}.
\end{split}   
\end{equation}

We claim that $u\in L^{\infty}([0,T]; L^2_{\text{ul}}(\R^2))$.  To see this, let $z\in\R^2$. 
By \cref{L2weakstarconv}, \cref{velunifbound}, and \cref{ln:uniform-bound-on-initial-data-of-sequence}, $$\| u \|_{L^{\infty}([0,T];L^2(B_1(z)))} 
\leq \sup_k \| u_k \|_{L^{\infty}([0,T];L^2(B_1(z)))} \leq C_0,$$
where $C_0$ depends on norms of the initial data and $T$ and is independent of $k$ and $z$.  Taking the supremum over $z$, we find that that $u\in L^{\infty}([0,T]; L^2_{\text{ul}}(\R^2))$.

By \cref{ln:uniform-bound-on-theta-k}, there also exists $\theta\in L^\infty([0,T]\times \R^2)$ such that, up to a subsequence, 
\begin{equation*}
\begin{split}
&\theta_k\rightarrow\theta \text{ weak-$\star$ in }L^{\infty}([0,T]\times \R^2).
\end{split}   
\end{equation*}
Finally, the fact that $\theta$ satisfies
$$\| \theta \|_{L^{\infty}([0,T]\times\R^2)} \leq \| \theta_0 \|_{L^{\infty}}$$
follows from \eqref{ln:uniform-bound-on-theta-k}.
\end{proof}

It follows immediately from \cref{prop:existence-of-limit-of-theta-u} that $u_k\to u$ in $\mathcal{D}'([0,T]\times\R^2)$.  In fact, $(u_k)$ converges to $u$ in $\mathcal{S}'([0,T]\times \R^2)$, as a consequence of the following lemma.
\begin{lemma}\label{Sconvergence}
 Let $(\theta_k,u_k)$ and $(\theta,u)$ be as in \cref{prop:existence-of-limit-of-theta-u}.  Then $u_k\rightarrow u$ in $\mathcal{S}'([0,T]\times\R^2)$.
\end{lemma}
\begin{proof}
Let $\rho\in\mathcal{S}([0,T]\times\R^2)$, and for $R>0$ let $a_R:\R^2\rightarrow\R$ be a smooth compactly supported function such that $a_R \equiv 1$ on $B_R(0)$, $0\leq a_R \leq 1$ on $B_{R+1}(0)$, and $a_R \equiv 0$ elsewhere. Then
\begin{equation}\label{Sconv}
\int_0^T\int_{\R^2} (u_k - u)\rho  = \int_0^T\int_{\R^2} (u_k - u)a_R\rho  + \int_0^T\int_{\R^2} (u_k - u)(1-a_R)\rho. 
\end{equation}  
For the second term on the right hand side, set $C_{k,R}= B_{(k+1)R}(0)\backslash B_{kR}(0)$ and write
\begin{equation*}
\begin{split}
&\left |\int_0^T\int_{\R^2} (u_k - u)(1-a_R)\rho \right | \leq  \int_{[0,T]\times B_R(0)^c} |u_k - u||\rho|  \leq \sum_{k\geq 1} \int_{[0,T]\times C_{k,R}} |u_k - u||\rho|\\
&\leq \sum_{k\geq 1}CkR \left(\| u_k \|_{L^1([0,T];L^2_{ul})} + \|u \|_{L^1([0,T];L^2_{ul})}\right) \| \rho  \|_{L^{\infty}([0,T];L^2(C_{k,R}))}\\
&\qquad \leq \frac{C}{R} \left(\| u_k \|_{L^1([0,T];L^2_{ul})} + \|u \|_{L^1([0,T];L^2_{ul})}\right) \leq \frac{C}{R},
\end{split}
\end{equation*}
where we used decay properties of $\rho\in\mathcal{S}([0,T]\times\R^2)$ and the uniform bound on $\| u_k \|_{L^2_{ul}}$, which follows from \eqref{L2ulestDSQG} and \eqref{ln:uniform-bound-on-initial-data-of-sequence}.  The factor of $kR$ in the second inequality above results from scaling; see for example Proposition 1.5 of \cite{Taniuchi}.  Now the second term on the right hand side of \cref{Sconv} can be made small by choosing $R$ sufficiently large.  Since $u_k \rightarrow u$ in $\mathcal{D}'([0,T]\times\R^2)$ by \cref{prop:existence-of-limit-of-theta-u}, given this choice of $R$, the first term on the right hand side of \cref{Sconv} is small for sufficiently large $k$.  Thus
$$\int_0^T\int_{\R^2} (u_k - u)\rho  \rightarrow 0.$$   
We conclude that $u_k\rightarrow u$ in $\mathcal{S}'([0,T]\times\R^2)$.
\end{proof}
    \begin{prop}\label{P:WeakTimeContinuityOfthetaSeq}
        Let $(\theta_k,u_k)$ be as in \cref{prop:existence-of-limit-of-theta-u}.
        For any $h \in C_c^2(\R^2)$ there exists $C_h > 0$ such that
        \begin{align*}
            \abs{\int_{\R^2} h(x) (\theta_k(t, x) - \theta_k(s, x)) \, dx}
                \le C_h \abs{t - s}
                \text{ for all } k \in \N, t, s \in [0, T].
        \end{align*}
    \end{prop}
    \begin{proof}
        Each $(\theta_k,u_k)$ satisfies
        \begin{align*}
            \theta_k(t, x) - \theta_k(s, x)
                &= -\int_s^t( \dv (\theta_k u_k) - (1/k)\Delta\theta_k)(r, x) \, dr.
        \end{align*}
        Multiplying through by a function $h \in C^2_c(\R^2)$ and integrating over $\R^2$ gives
        \begin{align*}
            &\abs{\int_{\R^2} h(x) (\theta_k(t, x) \, dx - \theta_k(s, x)) \, dx} \\
                &\qquad= \abs{\int_s^t\int_{\R^2} (\nabla h \cdot  \theta_k u_k + (1/k)\Delta h \theta_k)(r,x) \, dr} \\
            &\qquad
            \le \abs{t - s}
                \norm{\nabla h}_{L^2} \norm{\theta_k}_{L^\iny([0, T] \times \R^2)}
                    \norm{u_k}_{L^\iny(0, T; L^2(\supp(\grad h)))} \\
            &\qquad\qquad
                +\abs{t - s} \norm{h}_{C^2} \norm{\theta_k}_{L^\iny([0, T] \times \R^2)} \\
            &\qquad
            \le C_h \abs{t - s}.
            \qedhere
        \end{align*}
    \end{proof}
The following is an immediate corollary of \cref{P:WeakTimeContinuityOfthetaSeq} and \cref{L:WeakStarSectionalConvergence}.
\begin{cor}\label{corollarythetatimecont}
Let $\theta$ be as in \cref{prop:existence-of-limit-of-theta-u}.  Then, after changing $\theta$ on a set of measure zero, for all $h\in L^1(\R^2)$ and all $t\in [0,T]$,
\begin{equation*}
\begin{split}
&\int_{\R^2} h(x) \theta(s,x) \, dx \rightarrow \int_{\R^2} h(x) \theta(t,x) \, dx \quad\text{  as }s\rightarrow t.
\end{split}
\end{equation*}
Moreover, for all $x \in \R^2$,
\begin{equation*}
\begin{split}
& \abs{\int_{\R^2} h(x) (\theta(s, x) - \theta(t, x)) \, dx}
                 \le C_h \abs{t - s}
                 \text{ for all } h \in C_c^2(\R^2).
\end{split}
\end{equation*}
\end{cor}
We are now in position to show that the limit ($\theta,u$), as given in \cref{prop:existence-of-limit-of-theta-u}, satisfies the homogeneous constitutive law in \cref{weaksolutiondef2}. 
\begin{lemma}\label{constlaw}
Let ($\theta,u$) be as in \cref{prop:existence-of-limit-of-theta-u}.  Then for a.e. $(t,x)\in [0,T]\times\R^2$, ($\theta,u$) satisfies the homogeneous constitutive law.
\end{lemma}
\begin{proof}
Note that by \cref{Sconvergence}, for each $j\in\Z$,
$$ \dot{\Delta}_j u_k \longrightarrow \dot{\Delta}_j u$$
in $\mathcal{S'}([0,T]\times\R^2)$.  Similarly, by weak-$\star$ convergence of ($\theta_k$) to $\theta$ in $L^{\infty} ([0,T]\times \R^2)$, 
$$(\dot{\Delta}_jK)\ast {\theta}_k \longrightarrow  (\dot{\Delta}_jK)\ast \theta$$
in $\mathcal{S'}([0,T]\times\R^2)$.
Thus, since for each $k$, $(\theta_k,u_k)$ satisfies the homogeneous constitutive law, and by uniqueness of weak limits, ($\theta,u$) satisfies $$\dot{\Delta}_j u(t,x) = (\dot{\Delta}_jK)\ast \theta(t,x)$$
for a.e. $(t,x)\in [0,T]\times\R^2$.
\end{proof}
We can in fact show that, after changing $\theta$ and $u$ on a set of measure zero, $\Del_j u$ and $(\Del_jK)\ast\theta$ are continuous and equal on $[0,T]\times\R^2$.
\begin{cor}\label{constlaweverywhere}
Let $(\theta,u)$ be as in \cref{prop:existence-of-limit-of-theta-u}.  Then, after changing $\theta$ and $u$ on a set of measure zero, we have for each $j\in \Z$, $(\Del_jK)\ast\theta \in C([0,T]\times\R^2)$ and 
$$\Del_j u(t,x) = (\Del_jK)\ast\theta(t,x) \qquad \text{for all } (t,x)\in [0,T]\times \R^2.$$
\end{cor}
\begin{proof}
Continuity of $(\Del_jK)\ast\theta$ with respect to the spatial variable holds because $\Del_j K \in \mathcal{S}(\R^2)$, while continuity in time, after changing $\theta$ on a set of measure zero, follows from \cref{corollarythetatimecont} with $h=\Del_j K$.  The equality in \cref{constlaweverywhere} then follows from \cref{constlaw}.  
\end{proof}

The following lemma will be useful when passing to the limit of the nonlinear term in the weak formulation.

\begin{lemma}\label{nablasntozerolemma}
Let $(\theta, u)$ be as in \cref{prop:existence-of-limit-of-theta-u}.  Then, up to a subsequence,  
$$ u(t) - u_0 = \lim_{N\rightarrow -\infty} (H_NK) \ast (\theta(t) - \theta_0)  $$
uniformly on compact subsets of $[0,T]\times \R^2$.
\end{lemma}
\begin{proof}
 By \cref{SNulimit}, there exists a function $A:[0,T]\rightarrow\R^2$ such that, up to a subsequence,
 $$S_N u(t,x) - S_N u_0(x)  \rightarrow A(t) - A(0)$$
 uniformly on compact subsets of $[0,T]\times \R^2$.  We claim that $A=0$.  To see this, we will show that $ S_N u(t,x) - S_N u_0(x)  \rightarrow 0$ pointwise on $[0,T]\times\R^2$.  
Note that by \cref{renormalized2} and our choice of $U_{\infty}\equiv 0$ in \cref{Uequals0}, for each fixed $k\in\N$,
\begin{equation}\label{conv1} u_k(t) - u_{k,0} = \lim_{N\rightarrow -\infty} (H_NK) \ast (\theta_k(t) - \theta_{k,0})
\end{equation}
uniformly on compact subsets of $[0,T]\times\R^2$, so that, by an argument similar to \cref{constlawapp}, 
\begin{equation}\label{ukconvtoU} \lim_{N\rightarrow -\infty } S_N (u_k - u_{k,0}) = 0
\end{equation}
uniformly on compact subsets of $[0,T]\times\R^2$.  Moreover, by \cref{Sconvergence}, for each fixed $N$, 
\begin{equation}\label{snuktosnu} S_N u_k \rightarrow S_N u 
\end{equation}
pointwise on $[0,T]\times\R^2$.

We write, for fixed $(t,x)$,
\begin{equation}\label{finallimit}
\begin{split}
&\qquad \lim_{N\rightarrow -\infty} [S_N( u(t,x) - u_0(x)) ]\\
& = \lim_{N\rightarrow -\infty}\lim_{k\rightarrow \infty}[S_N( u(t,x) - u_0(x)) - S_N( u_k(t,x) - u_{k,0}(x))] \\
&  +\lim_{N\rightarrow -\infty}\lim_{k\rightarrow \infty}  [S_N( u_k(t,x) - u_{k,0}(x)) ]\\
& = \lim_{N\rightarrow -\infty}\lim_{k\rightarrow \infty}  [S_N( u_k(t,x) - u_{k,0}(x)) ],
\end{split}
\end{equation}
where we used \cref{snuktosnu} to get the last equality.  To see that the last limit in \cref{finallimit} equals zero, note that it follows from \cref{IH12} and \cref{ln:uniform-bound-on-initial-data-of-sequence} that the limit in \cref{lowfreqto0uniform} is uniform in $k$.  Thus $S_N( u_k(t,x) - u_{k,0}(x)) $ converges to zero as $N\rightarrow -\infty$ at a rate which is uniform in $k$, allowing us to interchange the limits and apply \cref{ukconvtoU}.

We conclude that the desired limit holds pointwise, and therefore uniformly on compact subsets of $[0,T]\times \R^2$.  This completes the proof.
\end{proof}
We are now in position to establish weak time continuity of $u$.

\begin{prop}\label{P:WeakTimeContinuityofu}
Let $(\theta,u)$ be as in \cref{prop:existence-of-limit-of-theta-u}.  After changing $u$ on a set of measure zero,  
for all compactly supported $g\in L^2(\R^2)$ and all $t\in [0,T]$, 
$$  \int_{\R^2} g(x) u(s,x) \, dx \rightarrow \int_{\R^2} g(x)u(t,x) \, dx \qquad \text{as }s\rightarrow t.$$
\end{prop}

\begin{proof}
First assume $g \in C^{\infty}_c(\R^2)$.  For fixed $N<0$ and $s,t\in[0,T]$, write 
\begin{equation}\label{ulimit3terms}
\begin{split}
&\int_{\R^2} g(x) (u(t,x)-u(s,x)) \, dx =   \int_{\R^2}g(x) S_N (u(t,x)-u(s,x)) \, dx\\
&\qquad + \int_{\R^2} g(x) H_N (u(t,x)-u(s,x)) \, dx.
\end{split}
\end{equation}

For the second term on the right hand side of \cref{ulimit3terms}, we apply the homogeneous constitutive law to write
\begin{equation*}
\begin{split}
&\int_{\R^2} g(x) H_N (u(t,x)-u(s,x)) \, dx = \int_{\R^2}g(x) (H_NK)\ast (\theta(t,x)-\theta(s,x)) \, dx  \\
&\qquad = \int_{\R^2} ((H_NK)\ast g(x))  (\theta(t,x)-\theta(s,x)) \, dx.
\end{split} 
\end{equation*}
By Bernstein's Lemma,  Young's inequality, and \cref{SNKbound},
\begin{equation}
\begin{split}
&\| (H_NK)\ast g \|_{L^1} \leq \sum_{j=N}^{-1} \| (\Del_jK)\ast g \|_{L^1} + \sum_{j\geq 0} \| (\Del_jK)\ast g \|_{L^1} \\
&\qquad \leq -N \|g \|_{L^1} + \sum_{j\geq 0} 2^{-j}\| (\Del_jK)\ast \nabla g \|_{L^1} \leq -CN\| g \|_{W^{1,1}}.
\end{split}
\end{equation}
We now invoke \cref{corollarythetatimecont} to conclude that
$$ \esslim_{s\rightarrow t} \int_{\R^2} g(x) H_N (u(t,x) - u(s,x)) \, dx = 0.$$
Now note that by \cref{nablasntozerolemma}, the first term on the right hand side of \cref{ulimit3terms} converges to $0$ as $N$ approaches negative infinity, where the convergence rate is uniform in $t$.  With this in mind, let $\epsilon>0$, and choose $N$ so that for almost every $s,t\in [0,T]$, 
\begin{equation*}
\abs{\int_{\R^2} g(x) S_N (u(t,x)-u(s,x)) \, dx} < \frac{\epsilon}{2}.
\end{equation*}
Given this $N$, let $\delta>$ be such that for almost every $t$,$s$ with $|t-s|<\delta$,
\begin{equation*}
\abs{\int_{\R^2} g(x) H_N (u(t,x) - u(s,x)) \, dx} < \frac{\epsilon}{2}.
\end{equation*}
Applying these inequalities to \cref{ulimit3terms} gives
$$\abs{\int_{\R^2} g(x) (u(t,x)-u(s,x)) \, dx}< \epsilon.$$

Now consider $g\in L^2(K)$ for $K\subset \R^2$ compact.  By density of $C^{\infty}_c(K)$ in $L^2(K)$ and the uniform-in-time bound
$$ \|u(t) - u(s)\|_{L^2(K)} \leq C(\|u(t)\|_{L^2_{ul}} + \| u(s)\|_{L^2_{ul}}) \leq C(T, \| u_0\|_{L^2_{ul}}, \|\theta_0\|_{L^{\infty}}),$$ 
we can conclude that
$$ \esslim_{s\rightarrow t} \int_{\R^2} g(x) (u(t,x) - u(s,x)) \, dx = 0. $$

\end{proof}

Having established the existence of weak limits, we turn to evaluating the convergence of the time derivative. 
\begin{lemma}\label{timeconvergence}
    Up to a subsequence, for every $\Psi\in C^{\infty}([0,T]\times\R^2)$,
$$ \int_0^T\int_{\R^2} \partial_t \theta_k(t,x) \Psi(t,x) \, dx \, dt \rightarrow \int_0^T\int_{\R^2} \partial_t \theta(t,x) \Psi(t,x) \, dx \, dt$$    as $k\rightarrow \infty$.
\end{lemma}
\begin{proof}
By \cref{max}, up to a subsequence, $(\theta_k)$ converges weak-$\star$ in $L^{\infty} ([0,T]\times \R^2)$.  Thus, for $\Psi\in C^{\infty}_c([0,T] \times \R^2)$, 
\begin{equation*}
\begin{split}
&\int_0^T\int_{\R^2} \partial_t \theta_k(t,x) \Psi(t,x) \, dx \, dt = - \int_0^T\int_{\R^2}  \theta_k(t,x) \partial_t \Psi(t,x) \, dx \, dt  
\end{split}
\end{equation*}
converges to $$ - \int_0^T\int_{\R^2}  \theta(t,x) \partial_t \Psi(t,x) \, dx \, dt = \int_0^T\int_{\R^2}  \partial_t\theta(t,x)  \Psi(t,x) \, dx \, dt,$$
as desired.
\end{proof}

\section{Existence of a weak solution}\label{S:Existence}

\noindent We devote this section to showing that $(\theta,u)$, given by \cref{prop:existence-of-limit-of-theta-u}, is a weak solution to ($SQG_0$) as in \cref{weaksolutiondef2}. The resulting series of Propositions leads, at the very end of this section, to the proof of \cref{T:MainResult}.

\medskip

\noindent {\bf Convergence of the nonlinear term.} As one would expect, the key step in showing that $(\theta,u)$ is a weak solution to ($SQG_0$) is establishing convergence of the nonlinear term. We address this convergence in the following proposition.
\begin{prop}\label{prop:nonlinear-term-convergence}
Let $(\theta_k, u_k)$ and $(\theta,u)$ be as in \cref{prop:existence-of-limit-of-theta-u}.  Then for any $\phi\in C^{\infty}_c([0,T]\times\R^2)$, up to a subsequence, 
\begin{equation}\label{nonlinearterm}
\int_0^T\int_{\R^2} \theta_k\nabla\phi\cdot u_k \, dx \, dt \rightarrow \int_0^T\int_{\R^2} \theta\nabla\phi \cdot u \, dx\, dt
\end{equation} 
as $k\rightarrow\infty$.
\end{prop}
\begin{proof}
Fix $N\in\Z$ with $N<0$.  We apply \cref{FourierSerfati2} with $\nu=1/k$ and $U_{\infty}\equiv 0$ to expand $u_k$ and write
\begin{equation} \label{ln:transport-term-expansion}
\begin{split}
&\int_{\R^2}\theta_k \nabla\phi\cdot  u_k \, dx =  \int_{\R^2} \theta_k\nabla\phi\cdot u_{k,0} \, dx - \int_{\R^2}\theta_k\nabla\phi\cdot\int_0^t (S_{N}\nabla K  \stardot(u_k\theta_k))  \, ds \, dx\\
&\qquad\qquad + \frac{1}{k}\int_{\R^2}\theta_k\nabla\phi\cdot\int_0^t (S_{N}\Delta K \ast \theta_k)  \, ds \, dx \\
&\qquad\qquad + \int_{\R^2}\theta_k\nabla\phi\cdot((H_{N}K) \ast (\theta_k(t) - \theta_{k,0})) \, dx\\
&\qquad\qquad=  \int_{\R^2} \theta_k\nabla\phi\cdot u_{k,0} \, dx+ \int_{\R^2} \theta_kl_{k,N} \, dx + \int_{\R^2} \theta_kh_{k,N} \, dx,
\end{split}
\end{equation} 
where, similar to \eqref{ln:def-of-I-and-II}, we set
\begin{equation} \label{ln:def-of-I-k-II-k}
\begin{split}
&l_{k,N}(t,x) = -\nabla\phi(t,x)\cdot\int_0^t \left( S_{N}\nabla K \stardot(u_k\theta_k) - 1/k S_{N}\Delta K \ast \theta_k\right)(s,x) \, ds,\\
&h_{k,N}(t, x) =  \nabla\phi(t,x)\cdot((H_{N}K) \ast (\theta_k(t) - \theta_{k,0}))(x).
\end{split}
\end{equation}
To complete the proof, we appeal to \crefrange{prop:initial-data-convergence}{P:convergence2} below, which address convergence of the three integrals on the right hand side of \eqref{ln:transport-term-expansion}.  We first apply \cref{prop:initial-data-convergence,prop:convergence1} to write
\begin{equation*}
\begin{split}
& \lim_{k\rightarrow\infty} \int_0^T\int_{\R^2}\theta_k \nabla\phi \cdot u_k \, dx \, dt=  \lim_{N\rightarrow-\infty}\lim_{k\rightarrow\infty}\int_0^T\int_{\R^2} \theta_k \nabla \phi \cdot (u_{k,0} -u_{0}) \, dx \, dt\\
&\qquad  +   \lim_{N\rightarrow-\infty}\lim_{k\rightarrow\infty}\int_0^T\int_{\R^2} \theta_k (\nabla \phi \cdot u_{0} + l_{k,N} + h_{k,N}) \, dx \, dt \\
& \qquad =  \lim_{N\rightarrow-\infty}\lim_{k\rightarrow\infty}\int_0^T\int_{\R^2} \theta_k (\nabla \phi \cdot u_{0} + l_{k,N} + h_{k,N})\, dx \, dt\\
&\qquad =\lim_{N\rightarrow-\infty}\lim_{k\rightarrow\infty} \int_0^T\int_{\R^2}\theta_k (\nabla \phi \cdot u_{0}  + h_{k,N}) \, dx \, dt.
\end{split}
\end{equation*} 
By weak-$\star$ convergence of $(\theta_k)$ to $\theta$ in $L^{\infty}([0,T]\times \R^2)$, \cref{P:convergence2}, and \cref{nablasntozerolemma}, 
\begin{equation*}
\begin{split}
&\qquad \lim_{k\rightarrow\infty} \int_0^T\int_{\R^2} \theta_k \nabla\phi \cdot u_k \, dx \, dt = \int_0^T\int_{\R^2} \theta \nabla\phi \cdot u_0 \, dx\, dt \\
&+\lim_{N\rightarrow -\infty} \int_0^T\int_{\R^2} \theta \nabla\phi \cdot \left( (H_NK)\ast (\theta(t) - \theta_0)(x) \right) \, dx\, dt  = \int_0^T\int_{\R^2} \theta \nabla\phi \cdot u_0 \, dx\, dt  \\
& +\int_0^T\int_{\R^2} \theta  \left(\lim_{N\rightarrow -\infty} \nabla\phi(t,x) \cdot ((H_NK)\ast (\theta(t) - \theta_0)(x)) \right) \, dx \, dt\\
&  = \int_0^T\int_{\R^2} \theta \nabla\phi \cdot u \, dx \, dt, 
\end{split}
\end{equation*}
where we used that the convergence in \cref{nablasntozerolemma} is uniform to move the limit inside the integral. 
\end{proof}

We now state and prove \crefrange{prop:initial-data-convergence}{P:convergence2}.  Throughout the proofs, we let $\psi\in C^{\infty}_c(\R^2)$ denote a smooth bump function with $\|\psi \|_{L^{\infty}} = 1$ satisfying 
$$\psi \equiv 1 \text{ on } \cup\{ \supp \nabla \phi(t,x): t\in [0,T]\}\subseteq\R^2. $$  In particular, we will apply many of the preceding estimates from \cref{S:Apriori,S:Existence} with $\eta = \psi$.  We begin with \cref{prop:initial-data-convergence}.
\begin{prop}\label{prop:initial-data-convergence}
Let $(\theta_k, u_k)$ and $(\theta,u)$ be as in \cref{prop:existence-of-limit-of-theta-u}.  Then for any $\phi\in C^{\infty}_c([0,T]\times \R^2)$,
\begin{equation}
    \lim_{k\to \infty} \int_0^T\int_{\R^2} \theta_k(t,x) \nabla \phi(t,x)\cdot (u_{k,0}(x) - u_0(x))\, dx \, dt=  0. 
\end{equation}
\end{prop}
\begin{proof}
      By \cref{ln:uniform-bound-on-theta-k} we have
    \begin{equation}
    \begin{split}
        &\abs{\int_0^T\int_{\R^2} \theta_k \nabla \phi (u_{k,0}- u_0) \, dx \, dt} \leq \|\theta_k \|_{L^1([0,T];L^\infty)} \sup_{t\in [0,T]}\int_{\R^2} |\psi(x) \nabla \phi(t,x)( u_0-u_{k,0})(x)| \, dx \\
        &\qquad \qquad\qquad\leq \|\theta_k \|_{L^1([0,T];L^\infty)} \sup_{t\in [0,T]}(\|\nabla \phi(t)\|_{L^2} \|\psi( u_0-u_{k,0})\|_{L^2} ) \\
        &\qquad \qquad\qquad\leq CT\|\theta_0 \|_{L^\infty} \|\psi( u_0-u_{k,0})\|_{L^2}. 
    \end{split}
    \end{equation}
    By \cref{prop:convergence-of-initial-data-u}, the claim follows.
\end{proof}

\begin{prop}\label{prop:convergence1}
Let $(\theta_k, u_k)$ be as in \cref{prop:existence-of-limit-of-theta-u}.  Then
\begin{equation}\label{convergence}
\begin{split}
&\lim_{N\rightarrow -\infty}\lim_{k\rightarrow\infty} \int_0^T\int_{\R^2} \theta_k(t,x)l_{k,N}(t,x) \, dx \, dt=0. 
\end{split}
\end{equation}    
\end{prop}
\begin{proof}
    
An application of \cref{IH12} to $l_{k,N}(t,x)$ with $\eta=\psi$ gives, for each $t\in [0,T]$,
\begin{equation*}
\| l_{k,N}(t,x) \|_{H^1_x} \leq C2^{N}\| \theta_{k,0}\|_{L^{\infty}}\left(\frac{t}{k} + \int_0^t \| u_k \|_{L^2_{ul}}\, ds\right).
\end{equation*}
Moreover, by \cref{L2ulestDSQG},
\begin{equation*}
\| u_k(t) \|_{L^2_{ul}} \leq C\left(\| u_{k,0} \|_{L^2_{ul}} + \|\theta_{k,0} \|_{L^{\infty}} + \frac{1}{k}\right)\exp \left( Ct\|\theta_{k,0}\|_{L^{\infty}}\right).
\end{equation*}
Thus, by \cref{ln:uniform-bound-on-initial-data-of-sequence}, for each fixed $N<0$ we get a uniform bound in $k$ on $\| l_{k,N}(t,x) \|_{H^1_x}$.  Since for each $k$, $l_{k,N}(t,x)$ has support contained in the support of $\nabla\phi(t,x)$, for each $N$ and each $t\in [0,T]$, we can apply Rellich's theorem to the sequence $( l_{k,N})_{k\in\N}$ and conclude that there exists a subsequence of $( l_{k,N})_{k\in\N}$, relabeled $( l_{k,N})_{k\in\N}$, converging strongly in $L^2(\R^2)$.  

We must show that there exists a single subsequence of $(l_{k,N})_{k\in\N}$ which converges strongly in $L^2(\R^2)$ for all $t\in [0,T]$ and for all $N<0$.  To this end, note that for each $k$ and $N<0$ and for fixed $\tau,t\in[0,T]$, 
\begin{equation}\label{Rellichapp}
\begin{split}
&\qquad\qquad \|l_{k,N}(t) - l_{k,N}(\tau)\|_{L^2}\\
&\leq  \|\nabla\phi(t,x)-\nabla\phi(\tau,x)\|_{L_x^\infty} \left\|\psi(x)\int_0^t \left( S_{N}\nabla K \stardot (u_k\theta_k) - (1/k) S_{N}\Delta K \ast \theta_k\right)(s,x) \, ds \right\|_{L^2_x} \\
&+  \|\nabla\phi(\tau,x)\|_{L_x^\infty}\left\|\psi(x)\int_{\tau}^t \left( S_{N}\nabla K \stardot (u_k\theta_k) - (1/k) S_{N}\Delta K \ast \theta_k\right)(s,x) \, ds \right\|_{L^2_x}.
\end{split}
\end{equation}
By \cref{lowfreqwithN}, \cref{L2ulestDSQG}, and smoothness of $\nabla\phi$, keeping in mind that $N<0$, we can bound the first term on the right hand side of \cref{Rellichapp} above by
$$ C(\|u_0\|_{L^2_{ul}}, \|\theta_0\|_{L^{\infty}}, T)|t-\tau|.$$
Moreover, following the proof of \cref{lowfreqwithN}, one can show that the second term on the right hand side of \cref{Rellichapp} can be bounded above by
$$ C(\|u_0\|_{L^2_{ul}}, \|\theta_0\|_{L^{\infty}}, T)|t-\tau|.$$
We conclude that 
\begin{equation}\label{singlesub} \|l_{k,N}(t) - l_{k,N}(\tau)\|_{L^2} \leq C(\|u_0\|_{L^2_{ul}}, \|\theta_0\|_{L^{\infty}}, T)|t-\tau|.
\end{equation}
We can use \cref{singlesub} and diagonalization to show that there exists a subsequence of $(l_{k,N})_{k\in\N}$ (relabeled $(l_{k,N})_{k\in\N}$), which converges strongly in $L^2(\R^2)$ at each $t\in [0,T]$ and for each $N<0$ to a function $F_N \in L^1([0,T]; L^2(\R^2))$ with spatial compact support.  Indeed, by \cref{lowfreqwithN}, for each $t\in [0,T]$,
\begin{equation}\label{FNbound} \| F_N(t)  \|_{L^2} \leq C(\| \theta_0\|_{L^{\infty}}, \| u_0 \|_{L^2_{ul}}, T)2^N.
\end{equation}

Now note that
\begin{equation}\label{Ikconvergestep1}
\begin{split}
& \lim_{N\rightarrow -\infty}\lim_{k\rightarrow\infty}\int_0^T\int_{\R^2} \theta_k(t,x)l_{k,N}(t,x) \, dx \, dt \\
&\qquad = \lim_{N\rightarrow -\infty}\lim_{k\rightarrow\infty} \int_0^T\int_{\R^2} \theta_k(t,x) (l_{k,N}(t,x)- F_N(t,x)) \, dx \, dt \\
&\qquad + \lim_{N\rightarrow -\infty}\lim_{k\rightarrow\infty} \int_0^T\int_{\R^2} (\theta_k(t,x) - \theta(t,x))F_N(t,x) \, dx \, dt \\
&\qquad + \lim_{N\rightarrow -\infty}\lim_{k\rightarrow\infty} \int_0^T\int_{\R^2} \theta(t,x)F_N(t,x) \, dx \, dt.  
\end{split}
\end{equation}
We show that the limit of each term on the right hand side of \cref{Ikconvergestep1} is zero.  Beginning with the second term, since for each $t\in [0,T]$, $F_N(t)$ has spatial compact support, $F_N$ belongs to $L^1([0,T]\times \R^2)$.  Thus by weak-$\star$ convergence of $(\theta_k)$ to $\theta$ in $L^{\infty}([0,T]\times \R^2)$, for each fixed $N$,
$$ \lim_{k\rightarrow\infty} \int_0^T\int_{\R^2} (\theta_k - \theta)F_N \, dx \, dt = 0.$$
For the first term on the right hand side of \cref{Ikconvergestep1}, note that
\begin{equation*}
\begin{split}
&\lim_{k\rightarrow\infty} \int_0^T\int_{\R^2} \theta_k (l_{k,N}- F_N) \, dx \, dt \leq \lim_{k\rightarrow\infty} \int_0^T \|\psi \theta_k\|_{L^2} \|l_{k,N}(t)- F_N(t)\|_{L^2}  \, dt\\
&\leq C\| \theta_0 \|_{L^{\infty}} \int_0^T \lim_{k\rightarrow\infty} \|l_{k,N}(t)- F_N(t)\|_{L^2}  \, dt = 0,
\end{split}
\end{equation*}
where we applied the Dominated Convergence Theorem to move the limit inside the integral.  Finally, for the third term on the right hand side of \cref{Ikconvergestep1}, we use \cref{FNbound} to write
\begin{equation*}
\begin{split}
&\abs{\int_0^T\int_{\R^2} \theta(t,x)F_N(t,x) \, dx \, dt} \leq  \| \psi\theta \|_{L^1([0,T];L^2)} \|F_N \|_{L^{\infty}([0,T]; L^2)} \\
&\qquad \leq   C(\|u_0\|_{L^2_{ul}}, \|\theta_0\|_{L^{\infty}}, T)2^N \rightarrow 0 
\end{split}
\end{equation*}
as $N\rightarrow -\infty$.  Combining the limits for the three terms, the claim follows.
\end{proof}

    \begin{prop}\label{P:convergence2}
    For each fixed $N \in \Z$ and any $\phi \in C_c^\iny([0, T) \times \R^2)$,
    \begin{equation}\label{e:KeyNonlinearConv}
    \begin{split}
    \lim_{k\rightarrow\infty}
    	\int_0^T &\int_{\R^2} \theta_k(t, x)
        \nabla\phi(t,x) \cdot ((H_N K) *
        	(\theta_k(t) - \theta_{k,0}))(x)
            \, dx \, dt \\
        &= \int_0^T\int_{\R^2} \theta(t, x)
            \nabla\phi(t,x) \cdot ((H_N K) * (\theta(t) - \theta_0))(x)
            \, dx \, dt.
    \end{split}
    \end{equation}
    \end{prop} 
    \begin{proof}
    We can write \cref{e:KeyNonlinearConv} in the form
    \begin{align*}
    	\lim_{k \to \iny}
			\brac{A(\theta_k, \theta_k - \theta_{k, 0})
				- A(\theta, \theta - \theta_0)}
			= 0,
    \end{align*}
    where
    \begin{align*}
    	A(f, g)
			&:= \int_0^T \int_{\R^2} f(t, x)
    	    \grad \phi(t,x) \cdot ((H_N K) * g(t))(x)
        	    \, dx \, dt,
    \end{align*}
    which we note is the integrated-in-time form of the operator in \cref{e:ABOperators}. We have,
    \begin{align}\label{e:ABreakdown}
    	\begin{split}
    	A(&\theta_k, \theta_k - \theta_{k, 0})
				- A(\theta, \theta - \theta_0) \\
			&= A(\theta_k, \theta_k - \theta_{k, 0})
				- A(\theta_k, \theta - \theta_0)
				+ A(\theta_k - \theta, \theta - \theta_0) \\
			&= A(\theta_k, \theta_k - \theta)
				+ A(\theta_k, \theta_0 - \theta_{k, 0})
				+ A(\theta_k - \theta, \theta - \theta_0) \\
			&=  A(\theta_k, \theta_0 - \theta_{k, 0})
				+ A(\theta_k - \theta, \theta - \theta_0) \\
			&\qquad\qquad
				+ A(\theta_k - \theta, \theta_k - \theta)
                + A(\theta, \theta_k - \theta).
		\end{split}
    \end{align}
    We bound each of the four terms on the right hand side of 
    \cref{e:ABreakdown} in turn.
    
	\smallskip\noindent$\boldsymbol{A(\theta_k,
		\theta_0 - \theta_{k, 0})}$:
	We have,
	\begin{align*}
		\abs{A(\theta_k, \theta_0 - \theta_{k, 0})}
			&\le \norm{\theta_k}_{L^\iny([0, T] \times \R^2)}
				\norm{\grad \phi}_{L^1(0, T; L^2)}
				\norm{H_N(u_0 - u_{k, 0})}_{L^2(\Sigma)},
	\end{align*}
	where $\Sigma$ is a compact set in $\R^2$ containing the support of $\grad \phi(t)$ for all $t \in [0, T]$. But, for all $k > N$, $H_N (u_0 - u_{k,0}) = H_N(H_k u_0) = H_k u_0 = u_0 - S_k u_0 \to 0$ in $L^2(\Sigma)$ by \cref{prop:convergence-of-initial-data-u}.

    \smallskip\noindent$\boldsymbol{A(\theta_k - \theta,
    	\theta - \theta_0)}$:
    Because $(\theta, u)$ satisfies the homogeneous constitutive law,  $(H_N K) * (\theta(t) - \theta_0)) = H_N (u(t) - u_0) \in L^2_{ul}(\R^2)$, and hence
	\begin{align*}
		\rho(t, x)
			&:= \grad \phi(t,x) \cdot
            	((H_N K) * (\theta(t) - \theta_0))(x)
			\in L^1([0, T] \times \R^2).
	\end{align*}
	\cref{prop:existence-of-limit-of-theta-u} gives
	\begin{align*}
		A(&\theta_k - \theta, \theta - \theta_0)
			= \int_0^T\int_{\R^2} (\theta_k - \theta) (t, x)
				\rho(t, x)
            \, dx \, dt
            \to 0.
	\end{align*}

    \smallskip\noindent$\boldsymbol{A(\theta_k - \theta,
    	\theta_k - \theta)}$:
    Define,
    \begin{align*}
        v_k
            &:= (H_N K) * \theta_k(t)
            = H_N u_k(t), \\
        v &
            := (H_N K) * \theta(t)
            = H_N u(t).
    \end{align*}
    The second expressions for $v_k$ and $v$ hold because $(\theta_k, u_k)$ and $(\theta, u)$ satisfy the homogeneous constitutive law.
	Then, $A(\theta_k - \theta, \theta_k - \theta) \to 0$ can
	be written as
    \begin{align}\label{e:KeyNonFinal}
         \lim_{k \to \iny} &\int_0^T \int_{\R^2}
         	(\theta_k - \theta)(t,x)
            \nabla\phi(t,x) \cdot (v_k - v)(t, x)
                \, dx \, dt
            = 0.
    \end{align}
    
    Define, for $M \ge N$,
    \begin{align}\label{e:HNM}
        H_N^M := \sum_{j = N}^M \Del_j,
    \end{align}
    and let $H_N^M = 0$ for $M < N$. Then $H_N^M K \in L^1(\R^2)$ because each $\smallnorm{\Del_j K}_{L^1} = \smallnorm{\Del_0 K}_{L^1}$.

    Fixing $N$, define
    \begin{align*}
        v_k^M &:= (H_N^M K) * \theta_k(t), \\
        v^M &:= (H_N^M K) * \theta(t).
    \end{align*}
    Using that $H_N^M K \in L^1(\R^2)$ and $\norm{\theta_k(t)} \le \norm{\theta_0}_{L^\iny}$, we know that $v^M \in L^\iny([0, T] \times \R^2)$ and $(v_k^M)_{k = 1}^\iny$ is bounded in $L^\iny([0, T] \times \R^2)$.

    By \cref{prop:existence-of-limit-of-theta-u}, $\theta_k \wstar \theta$ in $L^\iny([0, T] \times \R^2)$ (here, and in what follows, we take subsequences without further comment). By \cref{L:WeakStarSectionalConvergence} with \cref{P:WeakTimeContinuityOfthetaSeq}, then, we know that $\theta_k(t) \wstar \theta(t)$ in $L^\iny(\R^2)$ for all $t \in [0, T]$. Hence, using that $(H_N^M K)(x - \cdot) \in L^1(\R^2)$, for all $(t, x) \in [0, T] \times \R^2$,
    \begin{align*}        
        v_k^M(t, x)
            &= \int_{\R^2} (H_N^M K)(x - y) \theta_k(t, y)
                \, dy \\
            &\to \int_{\R^2} (H_N^M K)(x - y) \theta(t, y)
                \, dy
            = v^M(t, x).
    \end{align*}

    Fixing both $N$ and $M \ge N$, we have
    \begin{align*}
        &\abs{\int_0^T \int_{\R^2} (\theta_k - \theta)(t, x)
            \grad \phi(t,x) \cdot (v_k^M - v^M)(t, x)
            \, dx \, dt } \\
        &\qquad
        \le \norm{\theta_k - \theta}_{L^\iny([0, T] \times \R^2)}
        \norm{\grad \phi}_{L^\iny([0, T] \times \R^2)}
        \smallnorm{v_k^M - v^M}_{L^1(\supp(\grad \phi))} \\
        &\qquad
        \le C \norm{v_k^M - v^M}_{L^1(\supp(\grad \phi))}.
    \end{align*} 
    As we noted, $(v_k^M)_{k = 1}^\iny$ is bounded in $L^\iny([0, T] \times \R^2)$ and $v^M \in L^\iny([0, T] \times \R^2)$, so $v_k^M - v^M$ is bounded above by an $L^1$ function on $\supp(\grad \phi)$: applying the dominated convergence theorem, the pointwise convergence of $v_k^M - v^M$ to 0
    yields
    \begin{align}\label{e:vkMLimitInk}
        \lim_{k \to \iny} \int_0^T &\int_{\R^2}
        	(\theta_k - \theta)(t,x)
            \nabla\phi(t,x) \cdot (v_k^M - v^M)(t, x)
                \, dx \, dt
            = 0.
    \end{align}

	It remains to show that \cref{e:vkMLimitInk} holds in the limit
	as $M \to \iny$---uniformly in $k$---giving \cref{e:KeyNonFinal}.
	Applying \cref{L:ffInt}
	gives
	\begin{align}\label{e:vkMLimitInM}
		\begin{split}
		&\abs{\int_0^T \int_{\R^2}
        	(\theta_k - \theta)(t, x)
            \grad \phi(t,x) \cdot (v_k^M - v^M
            		- (v_k - v))(t, x)
                \, dx \, dt} \\
			&\qquad
			= \abs{\int_0^T \int_{\R^2}
	        	(\theta_k - \theta)(t, x)
    	        \grad  \phi(t,x) \cdot (H_{M + 1} K) *
	        		(\theta_k - \theta))(t, x)
            	    \, dx \, dt} \\
			&\qquad
			\le C \norm{\theta_k - \theta}_
					{L^\iny([0, T] \times \R^2)}^2
				2^{-M}
			\le C 2^{-M},
		\end{split}
	\end{align}
    independent of $k$.
	What we have shown is that
	\begin{align*}
		&\abs{A(\theta_k - \theta, \theta_k - \theta)}
            \le L_{k, M} + C 2^{-M},
	\end{align*}
	where $L_{k, M}$ is the magnitude of the integral in \cref{e:vkMLimitInk}.
    Let $\eps > 0$
	and choose $M$ large enough that $C 2^{-M} < \eps/2$.
	Then choose $n$ large enough that $L_{k, M}  \le \eps/2$ for all
	$k \ge n$ for the given choice of $M$. Then
	$\abs{A(\theta_k - \theta, \theta_k - \theta)} < \eps$ for all $k \ge n$,
	showing that $\abs{A(\theta_k - \theta, \theta_k - \theta)} \to 0$
	as $k \to \iny$.

    \medskip\noindent$\boldsymbol{A(\theta, \theta_k - \theta)}$:
    Recalling that $N$ is fixed, for any $M \ge N$ we can write $H_N = H_N^M + H_M$. Then, applying \cref{L:ffInt}, $A(\theta, \theta_k - \theta) = I_{k, M} + R_M$, where
    \begin{align*}
        I_{k, M}
			:= \int_0^T &\int_{\R^2} ((H_N^M K) * \theta(t, x))  \cdot \grad \phi(t,x)
                     (\theta_k - \theta)(t, x)  \, dx \, dt \\
            &+ \int_0^T \int_{\R^2} ((H_M K) * \theta(t, x))  \cdot \grad \phi(t,x)
                     (\theta_k - \theta)(t, x)  \, dx \, dt
    \end{align*}
    and $\abs{R_M} \le C 2^{-M}$ independently of $k$.
    But, $(H_M K) * \theta = H_M u \in L^{\infty}([0,T];L^2_{ul}(\R^2))$ by \cref{SnBound}, so $((H_M K) * \theta(t, x))  \cdot \grad \phi(t,x) \in L^1([0,T]\times\R^2)$, as also is $((H_N^M K) * \theta(t, x))  \cdot \grad \phi(t,x)$, since $H_M^N K$ is Schwartz-class.
    
    Hence, for any given $M$, $I_{k, M} \to 0$ as $k \to \iny$, though at an unknown rate.
    
    Let $\eps > 0$ and choose $M$ large enough that $\abs{R_M} < \eps/2$.
	Then choose $n$ large enough that $I_{k, M} \le \eps/2$ for all
	$k \ge n$ for the given choice of $M$. Then
	$\abs{A(\theta, \theta_k - \theta)} < \eps$ for all $k \ge n$,
	showing that $\abs{A(\theta, \theta_k - \theta)} \to 0$
	as $k \to \iny$. 
\end{proof}

With \cref{prop:nonlinear-term-convergence} in hand, we show that $(u, \theta)$ are, in fact, our desired weak solution.

\begin{prop}\label{prop:fourier-serfati-for-theta-u}
    Let $(\theta_k,u_k)$ and $(\theta,u)$ be as in \cref{prop:existence-of-limit-of-theta-u}.  Then $(\theta,u)$  satisfies \cref{weaksolutiondef2}.
\end{prop}
\begin{proof}
First note that ($\theta_k, u_k$) satisfies
$$\partial_t\theta_k + u_k\cdot\nabla \theta_k = \frac{1}{k} \Delta \theta_k$$
pointwise on $[0,T]\times\R^2$, so it satisfies
\begin{equation}
\begin{split}
&\int_{0}^T\int_{\R^2} \partial_t \phi \theta_k \, dx \, dt   +\int_{0}^T\int_{\R^2} \nabla \phi \cdot u_k\theta_k \, dx \, dt = -\frac{1}{k}\int_{0}^T\int_{\R^2} \Delta \phi \theta_k \, dx \, dt\\
&\qquad\qquad + \int_{\R^2}  \phi(T)\theta_{k}(T)\, dx-\int_{\R^2}  \phi(0)\theta_{k,0}\, dx
\end{split}
\end{equation}
for all $\phi\in C_c^{\infty}([0,T]\times \R^2)$.  By \cref{timeconvergence},
$$ \int_{0}^T\int_{\R^2} \partial_t \phi \theta_k \, dx \, dt \longrightarrow \int_{0}^T\int_{\R^2} \partial_t \phi \theta \, dx \, dt.  $$
By \cref{prop:nonlinear-term-convergence}, 
$$\int_{0}^T\int_{\R^2} \nabla \phi \cdot u_k\theta_k \, dx \, dt \longrightarrow \int_{0}^T\int_{\R^2} \nabla \phi \cdot u\theta \, dx \, dt. $$  
By \cref{L:WeakStarSectionalConvergence} and \cref{P:WeakTimeContinuityOfthetaSeq},
$$ \int_{\R^2}  \phi(T)\theta_{k}(T)\, dx \longrightarrow \int_{\R^2} \phi(T)\theta(T) \, dx.$$
By weak-$\star$ convergence of $(\theta_{k,0})$ to $\theta_0$ in $L^{\infty}(\R^2)$,
$$ \int_{\R^2}  \phi(0)\theta_{k,0}\, dx \longrightarrow \int_{\R^2} \phi(0) \theta_{0}\, dx.$$
Finally, by \cref{max}, 
$$\frac{1}{k}\int_{0}^T\int_{\R^2} \Delta \phi \theta_k \, dx \, dt \leq \frac{C}{k}\|\theta_0\|_{L^{\infty}} \longrightarrow 0.$$
By \cref{constlaw}, ($\theta,u$) satisfies the homogeneous constitutive law in \cref{weaksolutiondef2}.  Since $(u_k)$ converges to $u$ in $\mathcal{S'}([0,T]\times \R^2)$ and div $u_k=0$ on $[0,T]\times\R^2$, it follows that div $u=0$ on $[0,T]\times\R^2$.  Thus, ($\theta,u$) satisfies \cref{weaksolutiondef2}.
\end{proof}

\begin{prop}\label{P:InviscidSerfatiID}
    Assume $(\theta,u)$ is as in \cref{prop:existence-of-limit-of-theta-u}.  Then $(\theta,u)$ satisfies the Serfati identity \cref{FourierSerfati2} with $\nu=0$ and $U_\iny \equiv 0$.
\end{prop}
\begin{proof}
Fix $N\in\Z$ and $t\in [0,T]$.  By the homogeneous constitutive law,
\begin{equation} \label{lowhigh}
\begin{split}
&(u(t) - u_0)(x) = S_N(u(t) - u_0)(x) + H_N(u(t) - u_0)(x)\\
&\qquad\qquad = S_N(u(t) - u_0)(x) + (H_NK)\ast (\theta(t) - \theta_0)(x).
\end{split}
\end{equation}
Let $(\theta_k,u_k)$ be as in \cref{prop:existence-of-limit-of-theta-u}.  Since $(\theta_k,u_k)$ satisfies \cref{FourierSerfati2}, we have
\begin{equation}\label{aprroxlowfreq}
S_N(u_k(t) - u_{k,0})(x) = \int_0^t ( \nabla S_{N}K\stardot (\theta_ku_k) - (1/k) \Delta S_N K \ast \theta_k  )(s,x)  \, ds.
\end{equation}
We take the limit in $\mathcal{D}'([0,T]\times\R^2)$ of each term in \cref{aprroxlowfreq}.

First, note that since $u_k\rightarrow u$ in $\mathcal{S}'([0,T]\times\R^2)$, 
\begin{equation}\label{pwlimit}
S_N u_k \rightarrow S_N u
\end{equation} in $\mathcal{D}'([0,T]\times\R^2)$.  Moreover, $S_Nu_{k,0} \rightarrow S_Nu_0$ in $\mathcal{D}'([0,T]\times\R^2)$, as $S_Nu_{k,0} = S_NS_ku_{0} = S_Nu_0 $ for $k$ sufficiently larger than $N$.

For the right hand side of \cref{aprroxlowfreq}, we will take the pointwise limit on $[0,T]\times\R^2$ and apply the Dominated Convergence Theorem. 

We first show that 
\begin{equation}\label{weaksolutionhardterm}
 \int_0^t  \nabla S_{N}K\stardot (\theta_ku_k)(s,x)  \, ds  \rightarrow  \int_0^t   \nabla S_{N}K\stardot(\theta u)(s,x) \, ds 
\end{equation}
pointwise on $[0,T]\times \R^2$ as $k\rightarrow\infty$.  Let $R>0$ and let $a_R$ be as in the proof of \cref{Sconvergence}.  Note that by, for example, (2.3) of \cite{Cozzi}, for all $\delta\in (0,1)$, there exists $C_{\delta}>0$ such that
$$ |\nabla S_N K(x)| \leq \frac{C_{\delta}}{(1+|x|)^{3-\delta}}.$$  Using this spatial decay, membership of $a_RS_NK$ to $C^{\infty}_c([0,T]\times\R^2)$, and \cref{prop:nonlinear-term-convergence}, we can apply an argument virtually identical to the proof of \cref{Sconvergence}, but with $a_R\rho$ replaced by $\nabla((a_R S_N K)(x-\cdot))$, $(1-a_R)\rho$ replaced by $\nabla(((1-a_R) S_N K)(x-\cdot))$, and $u_k$ and $u$ replaced by $\theta_ku_k$ and $\theta u$, respectively.  The convergence in \cref{weaksolutionhardterm} follows for each $(t,x)\in [0,T]\times\R^2$.     

For the dissipation term, observe that by Young's inequality and the uniform $L^{\infty}$ bound on $(\theta_k)$,
\begin{equation}\label{laplaceterm}
\abs{\int_0^t(1/k)\Delta S_{N}K\ast\theta_k(s,x) \, ds} \leq t(1/k)\|\Delta S_{N}K \|_{L^1} \|\theta_0\|_{L^{\infty}}   \rightarrow 0  
\end{equation}
as $k\rightarrow\infty$.

One can quickly check, again using the proof of \cref{Sconvergence} and \cref{laplaceterm}, that
\begin{equation*}
\begin{split}
&\left\| \int_0^t \nabla S_N K \stardot (\theta_k u_k) \, ds \right\|_{L^{\infty}([0,T]\times\R^2)} \leq C,\\
& \left\| \int_0^t(1/k)\Delta S_{N}K\ast\theta_k \, ds \right\|_{L^{\infty}([0,T]\times\R^2)} \leq C
\end{split}
\end{equation*}
uniformly in $k$,
allowing us to apply the Dominated Convergence Theorem and conclude that the right hand side of \cref{aprroxlowfreq} converges in $\mathcal{D}'([0,T]\times\R^2)$ to $\int_0^t  \nabla S_{N}K\stardot (\theta u) (s,x)  \, ds.$  Thus, 
$$S_N(u(t) - u_{0})(x) = \int_0^t  \nabla S_{N}K\stardot (\theta u) (s,x)  \, ds$$
pointwise almost everywhere.  Substituting this identity into \cref{lowhigh} gives \cref{FourierSerfati2} with $\nu=0$ and $U_\iny \equiv 0$.
\end{proof}

\begin{proof}[\textbf{Proof of \cref{T:MainResult}}]
    Follows from combining
    \cref{prop:fourier-serfati-for-theta-u}, \ref{P:InviscidSerfatiID}, and \ref{P:SQG0Nonunique}. 
\end{proof}

\begin{remark}\label{R:TimeContinuity}
    Time continuity of $\theta$ as in \cref{corollarythetatimecont} follows directly for our weak solutions from that corollary, but also directly from the alternate formulation of a weak solution noted in \cref{R:WeakerTestFunctions}. Time continuity of $u$ would then follow as in \cref{P:WeakTimeContinuityofu}.
\end{remark}

\begin{remark}\label{R:InviscidSerfatiID}
    Because we chose to set $U_\iny \equiv 0$ in \cref{Uequals0}, the Serfati identity we obtained in \cref{P:InviscidSerfatiID} and hence in \cref{T:MainResult} has $U_\iny \equiv 0$. Had we chosen another continuous value of $U_\iny$ in  \cref{Uequals0}---the same value for each $(\theta_k, u_k)$---that $U_\iny$ would be added to the right side of the Serfati identity in \cref{e:SerfatiWeak}. The continuity of $U_\iny$ is required to obtain the weak continuity in time of the velocity, as given in \cref{P:WeakTimeContinuityofu}.
\end{remark}

\section*{Acknowledgments} DMA is grateful to the National Science Foundation
for support through grant DMS-2307638. EC is grateful to the Simons Foundation
for support through grant 429578.  

\section*{Data Availability} No datasets were generated or analyzed during this study.

\bibliography{Refs.bib}

@book{dragomir2003some,
    AUTHOR = {Dragomir, S.S.},
     TITLE = {Some {G}ronwall type inequalities and applications},
 PUBLISHER = {Nova Science Publishers, Inc., Hauppauge, NY},
      YEAR = {2003},
     PAGES = {viii+193},
      ISBN = {1-59033-827-8},
   MRCLASS = {34-01 (34A40 34C11 34D20 45-02)},
  MRNUMBER = {2016992},
MRREVIEWER = {B.\ G.\ Pachpatte},
}

@unpublished{AACK25,
  title={Smooth non-decaying solutions to the 2{D} dissipative quasi-geostrophic equations},
  author={Ambrose, D.M. and Aschoff, R. and Cozzi, E. and Kelliher, J.P.},
  year={2025},
  note={Preprint.  arXiv:2508.10254},
}

@article{ACEK,
    AUTHOR = {Ambrose, D.M. and Cozzi, E. and Erickson, D. and
              Kelliher, J.P.},
     TITLE = {Existence of solutions to fluid equations in {H}\"older and
              uniformly local {S}obolev spaces},
   JOURNAL = {J. Differential Equations},
  FJOURNAL = {Journal of Differential Equations},
    VOLUME = {364},
      YEAR = {2023},
     PAGES = {107--151},
      ISSN = {0022-0396,1090-2732},
   MRCLASS = {76B03 (35A01 35Q31 35Q35)},
  MRNUMBER = {4568776},
MRREVIEWER = {Da-Wen\ Deng},
       DOI = {10.1016/j.jde.2023.03.019},
       URL = {https://doi-org.ezproxy2.library.drexel.edu/10.1016/j.jde.2023.03.019},
}

@book{Chemin1,
    AUTHOR = {Chemin, J.-Y.},
     TITLE = {Perfect incompressible fluids},
    SERIES = {Oxford Lecture Series in Mathematics and its Applications},
    VOLUME = {14},
      NOTE = {Translated from the 1995 French original by Isabelle Gallagher
              and Dragos Iftimie},
 PUBLISHER = {The Clarendon Press, Oxford University Press, New York},
      YEAR = {1998},
     PAGES = {x+187},
      ISBN = {0-19-850397-0},
   MRCLASS = {76B47 (35-02 35Q35 76-02)},
  MRNUMBER = {1688875},
}

@article{Taniuchi,
    AUTHOR = {Taniuchi, Y.},
     TITLE = {Uniformly local {$L^p$} estimate for 2-{D} vorticity equation
              and its application to {E}uler equations with initial
              vorticity in {${\bf bmo}$}},
   JOURNAL = {Comm. Math. Phys.},
  FJOURNAL = {Communications in Mathematical Physics},
    VOLUME = {248},
      YEAR = {2004},
    NUMBER = {1},
     PAGES = {169--186},
      ISSN = {0010-3616,1432-0916},
   MRCLASS = {76B03 (35Q35)},
  MRNUMBER = {2104609},
MRREVIEWER = {Rapha\"el\ Danchin},
       DOI = {10.1007/s00220-004-1095-6},
       URL = {https://doi-org.ezproxy2.library.drexel.edu/10.1007/s00220-004-1095-6},
}

@article{Cozzi,
  author    = {E. Cozzi},
  title     = {Solutions to the 2D Euler equations with velocity unbounded at infinity},
  journal   = {Journal of Mathematical Analysis and Applications},
  volume    = {423},
  number    = {1},
  pages     = {144--161},
  year      = {2015}
}

@article{cordoba2004maximum,
    AUTHOR = {C\'ordoba, A. and C\'ordoba, D.},
     TITLE = {A maximum principle applied to quasi-geostrophic equations},
   JOURNAL = {Comm. Math. Phys.},
  FJOURNAL = {Communications in Mathematical Physics},
    VOLUME = {249},
      YEAR = {2004},
    NUMBER = {3},
     PAGES = {511--528},
      ISSN = {0010-3616,1432-0916},
   MRCLASS = {76B03 (26A33 35Q35 76U05 86A05)},
  MRNUMBER = {2084005},
MRREVIEWER = {Luigi\ Carlo\ Berselli},
       DOI = {10.1007/s00220-004-1055-1},
       URL = {https://doi-org.ezproxy2.library.drexel.edu/10.1007/s00220-004-1055-1},
}

@article {KBounded,
    AUTHOR = {Kelliher, J.P.},
     TITLE = {A characterization at infinity of bounded vorticity, bounded
              velocity solutions to the 2{D} {E}uler equations},
   JOURNAL = {Indiana Univ. Math. J.},
  FJOURNAL = {Indiana University Mathematics Journal},
    VOLUME = {64},
      YEAR = {2015},
    NUMBER = {6},
     PAGES = {1643--1666},
      ISSN = {0022-2518,1943-5258},
   MRCLASS = {76B47},
  MRNUMBER = {3436230},
MRREVIEWER = {Rapha\"el\ Danchin},
       DOI = {10.1512/iumj.2015.64.5717},
       URL = {https://doi-org.ezproxy2.library.drexel.edu/10.1512/iumj.2015.64.5717},
}

@article {KukavicaBoundedNS,
    AUTHOR = {Kukavica, I.},
     TITLE = {On local uniqueness of weak solutions of the {N}avier-{S}tokes
              system with bounded initial data},
   JOURNAL = {J. Differential Equations},
  FJOURNAL = {Journal of Differential Equations},
    VOLUME = {194},
      YEAR = {2003},
    NUMBER = {1},
     PAGES = {39--50},
      ISSN = {0022-0396,1090-2732},
   MRCLASS = {35Q30 (76D03 76D05)},
  MRNUMBER = {2001028},
MRREVIEWER = {J\"urgen\ Socolowsky},
       DOI = {10.1016/S0022-0396(03)00153-0},
       URL = {https://doi-org.ezproxy2.library.drexel.edu/10.1016/S0022-0396(03)00153-0},
}

@book {FMRT,
    AUTHOR = {Foias, C. and Manley, O. and Rosa, R. and Temam, R.},
     TITLE = {Navier-{S}tokes equations and turbulence},
    SERIES = {Encyclopedia of Mathematics and its Applications},
    VOLUME = {83},
 PUBLISHER = {Cambridge University Press, Cambridge},
      YEAR = {2001},
     PAGES = {xiv+347},
      ISBN = {0-521-36032-3},
   MRCLASS = {76-02 (35Q30 37L30 37N10 76D05 76D06 76F05 76F20)},
  MRNUMBER = {1855030},
MRREVIEWER = {Xiaoming\ Wang},
       DOI = {10.1017/CBO9780511546754},
       URL = {https://doi-org.ezproxy2.library.drexel.edu/10.1017/CBO9780511546754},
}

@article {Berselli,
    AUTHOR = {Berselli, L.C.},
     TITLE = {Vanishing viscosity limit and long-time behavior for 2{D}
              quasi-geostrophic equations},
   JOURNAL = {Indiana Univ. Math. J.},
  FJOURNAL = {Indiana University Mathematics Journal},
    VOLUME = {51},
      YEAR = {2002},
    NUMBER = {4},
     PAGES = {905--930},
      ISSN = {0022-2518,1943-5258},
   MRCLASS = {35Q35 (35B40 35B41 37L30 76U05)},
  MRNUMBER = {1947863},
       DOI = {10.1512/iumj.2002.51.2075},
       URL = {https://doi-org.ezproxy2.library.drexel.edu/10.1512/iumj.2002.51.2075},
}

@article {Marchand,
    AUTHOR = {Marchand, F.},
     TITLE = {Existence and regularity of weak solutions to the
              quasi-geostrophic equations in the spaces {$L^p$} or {$\dot
              H^{-1/2}$}},
   JOURNAL = {Comm. Math. Phys.},
  FJOURNAL = {Communications in Mathematical Physics},
    VOLUME = {277},
      YEAR = {2008},
    NUMBER = {1},
     PAGES = {45--67},
      ISSN = {0010-3616,1432-0916},
   MRCLASS = {76B03 (35D05 35D10 35Q35 76B60 76D03 76U05)},
  MRNUMBER = {2357424},
MRREVIEWER = {Luigi\ Carlo\ Berselli},
       DOI = {10.1007/s00220-007-0356-6},
       URL = {https://doi-org.ezproxy2.library.drexel.edu/10.1007/s00220-007-0356-6},
}

@article {Serfati,
    AUTHOR = {Serfati, P.},
     TITLE = {Solutions {$C^\infty$} en temps, {$n$}-{$\log$} {L}ipschitz
              born\'ees en espace et \'equation d'{E}uler},
   JOURNAL = {C. R. Acad. Sci. Paris S\'er. I Math.},
  FJOURNAL = {Comptes Rendus de l'Acad\'emie des Sciences. S\'erie I.
              Math\'ematique},
    VOLUME = {320},
      YEAR = {1995},
    NUMBER = {5},
     PAGES = {555--558},
      ISSN = {0764-4442},
   MRCLASS = {35Q30 (76C99)},
  MRNUMBER = {1322336},
MRREVIEWER = {Rodolfo\ Salvi},
}

@article {AKLN2015,
    AUTHOR = {Ambrose, D.M. and Kelliher, J.P. and Lopes Filho,
              M.C. and Nussenzveig Lopes, H.J.},
     TITLE = {Serfati solutions to the 2{D} {E}uler equations on exterior
              domains},
   JOURNAL = {J. Differential Equations},
  FJOURNAL = {Journal of Differential Equations},
    VOLUME = {259},
      YEAR = {2015},
    NUMBER = {9},
     PAGES = {4509--4560},
      ISSN = {0022-0396,1090-2732},
   MRCLASS = {35Q31 (76B03)},
  MRNUMBER = {3373413},
MRREVIEWER = {Christophe\ Lacave},
       DOI = {10.1016/j.jde.2015.06.001},
       URL = {https://doi-org.ezproxy2.library.drexel.edu/10.1016/j.jde.2015.06.001},
}

@article {Wu1997,
    AUTHOR = {Wu, J.},
     TITLE = {Inviscid limits and regularity estimates for the solutions of
              the {$2$}-{D} dissipative quasi-geostrophic equations},
   JOURNAL = {Indiana Univ. Math. J.},
  FJOURNAL = {Indiana University Mathematics Journal},
    VOLUME = {46},
      YEAR = {1997},
    NUMBER = {4},
     PAGES = {1113--1124},
      ISSN = {0022-2518,1943-5258},
   MRCLASS = {35Q35 (76C99 76F99 76U05)},
  MRNUMBER = {1631560},
       DOI = {10.1512/iumj.1997.46.1275},
       URL = {https://doi-org.ezproxy2.library.drexel.edu/10.1512/iumj.1997.46.1275},
}

@article {CozziKelliher2019,
    AUTHOR = {Cozzi, E. and Kelliher, J.P.},
     TITLE = {Well-posedness of the 2{D} {E}uler equations when velocity
              grows at infinity},
   JOURNAL = {Discrete Contin. Dyn. Syst.},
  FJOURNAL = {Discrete and Continuous Dynamical Systems},
    VOLUME = {39},
      YEAR = {2019},
    NUMBER = {5},
     PAGES = {2361--2392},
      ISSN = {1078-0947,1553-5231},
   MRCLASS = {35Q31 (35B30 76B03)},
  MRNUMBER = {3927517},
MRREVIEWER = {Francesca\ Brini},
       DOI = {10.3934/dcds.2019100},
       URL = {https://doi-org.ezproxy2.library.drexel.edu/10.3934/dcds.2019100},
}

@book {Resnick,
    AUTHOR = {Resnick, S.G.},
     TITLE = {Dynamical problems in non-linear advective partial
              differential equations},
      NOTE = {Thesis (Ph.D.)--The University of Chicago},
 PUBLISHER = {ProQuest LLC, Ann Arbor, MI},
      YEAR = {1995},
     PAGES = {76},
   MRCLASS = {99-05},
  MRNUMBER = {2716577},
       URL =
              {http://gateway.proquest.com.ezproxy2.library.drexel.edu/openurl?url_ver=Z39.88-2004&rft_val_fmt=info:ofi/fmt:kev:mtx:dissertation&res_dat=xri:pqdiss&rft_dat=xri:pqdiss:9542767},
}

@article {ConstantinMajdaTabak1994A,
    AUTHOR = {Constantin, P. and Majda, A.J. and Tabak, E.G.},
     TITLE = {Singular front formation in a model for quasigeostrophic flow},
   JOURNAL = {Phys. Fluids},
  FJOURNAL = {Physics of Fluids},
    VOLUME = {6},
      YEAR = {1994},
    NUMBER = {1},
     PAGES = {9--11},
      ISSN = {1070-6631,1089-7666},
   MRCLASS = {76C99 (76C15 76M25 86A10)},
  MRNUMBER = {1252829},
       DOI = {10.1063/1.868050},
       URL = {https://doi-org.ezproxy2.library.drexel.edu/10.1063/1.868050},
}

@article {ConstantinMajdaTabak1994B,
    AUTHOR = {Constantin, P. and Majda, A.J. and Tabak, E.},
     TITLE = {Formation of strong fronts in the {$2$}-{D} quasigeostrophic
              thermal active scalar},
   JOURNAL = {Nonlinearity},
  FJOURNAL = {Nonlinearity},
    VOLUME = {7},
      YEAR = {1994},
    NUMBER = {6},
     PAGES = {1495--1533},
      ISSN = {0951-7715,1361-6544},
   MRCLASS = {76U05 (35Q35 76C15 76M25 86A10)},
  MRNUMBER = {1304437},
MRREVIEWER = {J.\ Thomas\ Beale},
       DOI = {10.1088/0951-7715/7/6/001},
       URL = {https://doi-org.ezproxy2.library.drexel.edu/10.1088/0951-7715/7/6/001},
}

@article {LuigiLatocca2026,
    AUTHOR = {De Rosa, L. and Latocca, M. and Park, J.},
     TITLE = {Global {E}xistence, {H}amiltonian {C}onservation and
              {V}anishing {V}iscosity for the {S}urface
              {Q}uasi-{G}eostrophic {E}quation},
   JOURNAL = {Arch. Ration. Mech. Anal.},
  FJOURNAL = {Archive for Rational Mechanics and Analysis},
    VOLUME = {250},
      YEAR = {2026},
    NUMBER = {5},
     PAGES = {Paper No. 74},
      ISSN = {0003-9527,1432-0673},
   MRCLASS = {35Q35 (35D30 35Q86 76B03)},
  MRNUMBER = {5111259},
       DOI = {10.1007/s00205-026-02231-2},
       URL = {https://doi-org.ezproxy2.library.drexel.edu/10.1007/s00205-026-02231-2},
}

@article {Ambrose2025,
    AUTHOR = {Ambrose, D.M.},
     TITLE = {The velocity field and {B}irkhoff-{R}ott integral for
              nondecaying, nonperiodic vortex sheets},
   JOURNAL = {SIAM J. Appl. Math.},
  FJOURNAL = {SIAM Journal on Applied Mathematics},
    VOLUME = {85},
      YEAR = {2025},
    NUMBER = {2},
     PAGES = {456--476},
      ISSN = {0036-1399,1095-712X},
   MRCLASS = {76B55 (42A50 76B07)},
  MRNUMBER = {4875664},
       DOI = {10.1137/24M164848X},
       URL = {https://doi-org.ezproxy2.library.drexel.edu/10.1137/24M164848X},
}

@article {Ciampa2022,
    AUTHOR = {Ciampa, G.},
     TITLE = {Energy conservation for 2{D} {E}uler with vorticity in
              {$L(\log L)^\alpha$}},
   JOURNAL = {Commun. Math. Sci.},
  FJOURNAL = {Communications in Mathematical Sciences},
    VOLUME = {20},
      YEAR = {2022},
    NUMBER = {3},
     PAGES = {855--875},
      ISSN = {1539-6746,1945-0796},
   MRCLASS = {35Q31},
  MRNUMBER = {4401505},
       DOI = {10.4310/CMS.2022.v20.n3.a10},
       URL = {https://doi-org.ezproxy2.library.drexel.edu/10.4310/CMS.2022.v20.n3.a10},
}

@article {CiampaCrippaSpirito2021,
    AUTHOR = {Ciampa, G. and Crippa, G. and Spirito, S.},
     TITLE = {Strong convergence of the vorticity for the 2{D} {E}uler
              equations in the inviscid limit},
   JOURNAL = {Arch. Ration. Mech. Anal.},
  FJOURNAL = {Archive for Rational Mechanics and Analysis},
    VOLUME = {240},
      YEAR = {2021},
    NUMBER = {1},
     PAGES = {295--326},
      ISSN = {0003-9527,1432-0673},
   MRCLASS = {76B47 (35Q31)},
  MRNUMBER = {4228862},
MRREVIEWER = {Tomasz\ Cie\'slak},
       DOI = {10.1007/s00205-021-01612-z},
       URL = {https://doi-org.ezproxy2.library.drexel.edu/10.1007/s00205-021-01612-z},
}
\bibliographystyle{plain}

\end{document}